\documentclass[11pt]{article}
\usepackage{diagrams}
\usepackage{amsmath}
\usepackage{amssymb,color,mathtools}
\usepackage{mathrsfs}
\usepackage{amscd}
\usepackage{amsthm}
\usepackage[matrix,frame,import,curve]{xypic}
\usepackage{dirtytalk}

\newtheorem{lemma}{Lemma}[section]
\newtheorem{theorem}[lemma]{Theorem}
\newtheorem{corollary}[lemma]{Corollary}

\newtheorem{prop}[lemma]{Proposition}
\theoremstyle{definition}
\newtheorem{definition}[lemma]{Definition}
\newtheorem{example}[lemma]{Example}

\newtheorem{remark}[lemma]{Remark}

\theoremstyle{remark}
\newtheorem*{proof*}{Proof}
\numberwithin{equation}{section}

\renewcommand{\th}{\theta}

\def\ep{{\epsilon}}

\renewcommand{\k}{\mathbf{k}}

\newcommand{\La}{\Lambda}
\newcommand{\sub}{\subset}

\newcommand{\De}{\Delta}

\def\Spec{{\bf {Spec}}}
\def\RHom{{\mathbf{R}\rm{Hom}}}

\def\D{\mathrm{D}}

\def\Ext{{\mathrm{Ext}}}

\def\Hom{{\mathrm{Hom}}}

\def\End{{\mathrm{End}}}

\def\perf{{\mathrm{Perf}}}
\def\Spec{{\mathrm{Spec\ }}}
\def\deg{{\mathrm{deg}}}
\def\dim{{\mathrm{dim}}}

\def\Res{{\mathrm{Res}}}

\def\rk{{\mathrm{rk}}}
\newcommand{\om}{\omega}
\newcommand{\OO}{{\mathcal O}}
\newcommand{\MM}{{\mathcal M}}
\newcommand{\id}{{\operatorname{id}}}

\def\ot{{\otimes}}

\newcommand{\und}{\underline}

\def\PP{{\mathbb P}}

\def\ZZ{{\mathbb Z}}
\def\CC{{\mathbb C}}

\def\NN{{\mathbb N}}
\def\RR{{\mathbb R}}
\def\HH{{\mathbb H}}
\def\AA{{\mathbb A}}
\def\GG{{\mathbb G}}

\newcommand{\we}{\wedge}
\def\cF{{\cal{F}}}
\newcommand{\FF}{{\mathcal F}}
\def\cE{{\cal{E}}}
\def\cO{{\cal{O}}}
\def\cC{{\cal{C}}}

\def\cM{{\cal{M}}}

\def\cA{{\cal{A}}}

\def\cV{{\cal{V}}}

\def\cD{{\cal{D}}}

\def\sX{{\mathscr{X}}}
\def\sY{{\mathscr{Y}}}

\newcommand{\pa}{\partial}
\newcommand{\la}{\lambda}

\newcommand{\GL}{{\operatorname{GL}}}
\newcommand{\Aut}{{\operatorname{Aut}}}

\newcommand{\lan}{\langle}
\newcommand{\ran}{\rangle}

\def\fm{{\mathfrak{m}}}

\def\Res{{\mathrm{Res}}}
\def\gr{{\rm{gr}}}

\def\p{{\prime}}

\def\ul{\underline}
\def\wt{\widetilde}

\def\Map{{\bf{Map}}}
\def\Vt{\underline{Vect}}
\def\dVt{\RR\underline{Vect}}

\def\Qcoh{\mathrm{Qcoh}}
\def\dPer{\RR{\underline{Perf}}}
\def\dePer{\RR\ep{\underline{Perf}}}
\def\dCx{\RR{\underline{Cplx}}}
\def\bCx{{\bf{Cplx}}}
\def\cdga{\mathbf{cdga}}

\newcommand{\de}{\delta}
\newcommand{\sspan}{\operatorname{span}}
\newcommand{\VV}{{\mathcal V}}

\title{Elliptic bihamiltonian structures from relative shifted Poisson structures}
\author{Zheng Hua, Alexander Polishchuk}\date{}

\begin{document}
\maketitle

\begin{abstract}
In this paper, generalizing the construction of \cite{HP1}, we equip the relative moduli stack of complexes over a Calabi-Yau fibration
(possibly with singular fibers) with a shifted Poisson structure. Applying this construction to the anticanonical
linear systems on surfaces, we get examples of compatible Poisson brackets on projective spaces extending Feigin-Odesskii Poisson brackets. 
Computing explicitly the corresponding compatible brackets coming from Hirzebruch surfaces, we recover the brackets defined by Odesskii-Wolf in 
\cite{OW}.
\end{abstract}


\section{Introduction}

Recall that a {\it bihamiltonian structure} is a pair of (linearly independent) Poisson bivectors
$\Pi_1,\Pi_2$ which are {\it compatible}, i.e., such that any linear combination of $\Pi_1$ and $\Pi_2$
is again Poisson. A fundamental result of Magri relates bihamiltonian structures to complete integrability \cite{Magri}.  

The main goal of this paper is to try to understand the geometry underlying bihamiltonian structures
extending the elliptic Feigin-Odesskii Poisson brackets. Recall that the latter are certain Poisson brackets $q_{n,k}(C)$ on the projective
space $\PP^{n-1}$ associated with an elliptic curve $C$ and a pair of relatively prime integers $n>k>0$ 
(see Sec.\ \ref{FO-br-sec}). These brackets were introduced by Feigin and Odesskii in \cite{FO95} and are supposed to arise as semiclassical limits
from Feigin-Odesskii elliptic algebras introduced in \cite{FO87} (for $k=1$ this is proved in \cite[Sec.\ 5.2]{HP1}).
Recently interesting examples of such bihamiltonian structures were constructed by
Odesskii-Wolf in \cite{OW} (improving earlier construction of Odesskii in \cite{Odesskii}): 
for every $n>2$ they constructed a $9$-dimensional subspace of compatible Poisson brackets on $\PP^{n-1}$
containing $q_{n,1}(C)$. Our results give a more conceptual construction of these compatible brackets, as well as some
generalizations involving $q_{n,k}(C)$ with $k>1$.

The main idea is to use the general setup of shifted Poisson structures on (derived) moduli stacks of complexes of vector bundles
over Calabi-Yau varieties considered in \cite{HP1}. In \cite{HP1} we showed that Feigin-Odesskii brackets appear in
this setup as classical shadows of natural $0$-shifted Poisson structures on the moduli stacks of two-term complexes over elliptic curves 
(in fact, this connection goes back to \cite{Pol-Poisson}).
In this paper we extend this setup by allowing the varieties to be singular Gorenstein and by considering a relative version.
More precisely, for a flat family of (possibly singular) $d$-Calabi-Yau varieties with an affine base, 
there is  a $(1-d)$-shifted Poisson structure on the relative stack of complexes (see Theorem \ref{existence-relative}).
We show that in the case of elliptic fibrations $\pi:C\to \PP^n$ such that $\om_{C/S}\simeq \pi^*\OO_{\PP^n}(1)$ this leads
to families of compatible Poisson brackets (see Theorem \ref{proj-base-thm}).

We then proceed to study families of anticanonical divisors on surfaces. We find a general construction starting from an exceptional bundle $\cV$ on a surface $X$,
such that $(\OO_X,\cV)$ is an exceptional pair, and leading to compatible brackets containing Feigin-Odesskii brackets
(see Theorem \ref{exc-pair-thm}). Considering appropriate line bundles on Hirzebruch surfaces we recover the $9$ compatible Poisson brackets of Odesskii-Wolf
containing $q_{n,1}(C)$. Proving that these are actually the same compatible brackets is a nontrivial computation that takes up Section \ref{computations-sec}.
These computations are based on the connection between the Poisson brackets $q_{n,k}(C)$ and certain Massey products. We calculate the relevant Massey products
using Szeg\"o kernels. 

We also discover some new examples of compatible Poisson brackets. Namely, we construct two infinite families of pairs $(n,k)$ for which each Feigin-Odesskii
bracket $q_{n,k}(C)$ is contained in a $10$-dimensional family of compatible Poisson brackets, namely, the pairs 
$$(3f_{2m-1},f_{2m-3}) \ \text{ for } m\ge 2, \ \text{ and } \ (3f_{2m-1},3f_{2m-1}-f_{2m-3}) \ \text{ for } m\ge 3,$$
where $(f_n)$ are Fibonacci numbers (see Proposition \ref{Fibonacci-prop}). For example, this gives a $10$-dimensional subspace of compatible Poisson brackets
on $\PP^5$ containing $q_{6,1}(C)$, which is a bit surprising given that the $9$-dimensional space of compatible brackets of Odesskii-Wolf on $\PP^5$ is maximal, i.e., 
is not contained in a bigger such space. This leads to a natural question how these two spaces are related.

Another new example we discover is that for every $n>k>1$ such that $n\equiv \pm 1 \mod (k)$, with odd $k$,
there exists a bihamiltonian structure on $\PP^{n-1}$ containing $q_{n,k}(C)$ (see Proposition \ref{biham5-prop}). 
In fact, in this example we get $5$ compatible brackets but we don't know how to prove their linear independence.

The natural question is whether for every relatively prime pair $(n,k)$ with $n>k+1$, the Feigin-Odesskii
bracket $q_{n,k}(C)$ extends to a bihamiltonian structure. We believe that our construction using exceptional bundles on surfaces 
in Theorem \ref{exc-pair-thm} should at least provide more examples of such pairs (if not all of them).

It is an interesting question whether bihamiltonian structures containing Feigin-Odesskii brackets lead to any interesting integrable systems.
We plan to address this question in a future work.

The paper is organized as follows. In Section \ref{FO-br-sec} we study Feigin-Odesskii Poisson brackets $q_{n,k}(C)$.
The first result here is the formula for the bracket in terms of a triple Massey product (see Lemma \ref{Poisson-Massey-lem}).
The second result of Section \ref{FO-br-sec}, which may be of independent interest, 
is that the isomorphism class of an elliptic curve $C$ can be recovered from $q_{n,k}(C)$
provided $n>k+1$ (see Theorem \ref{reconstr-thm}). We prove this by studying the locus where the rank of the Poisson bivector drops compare to 
the generic rank. In Section \ref{singular-source-sec} we generlize the construction of a shifted Poisson structure on the moduli of complexes
over a smooth Calabi-Yau variety from \cite{HP1} to the case of families of not necessarily smooth Calabi-Yau varieties (see Theorem \ref{existence-relative}).
In Section \ref{CY-curves-sec} we specialize to families of CY-curves. Considering a relative version of Feigin-Odesskii
Poisson brackets, under appropriate assumptions we get collections of compatible Poisson brackets on projective
spaces (see Theorem \ref{proj-base-thm}). We then show that compatible Poisson brackets
arise from the linear system of anticanonical divisors in a smooth projective surface $X$ and an exceptional pair $(\OO_X,\VV)$ (see Theorem
\ref{exc-pair-thm}). We consider examples corresponding to such exceptional pairs on some del Pezzo surfaces and Hirzebruch surfaces.
Finally, in Section \ref{computations-sec} we show how to compute our Poisson brackets in terms of Szeg\"o kernels and deduce
that our construction, applied to exceptional pairs on Hirzebruch surfaces, recovers the compatible Poisson brackets of Odesskii-Wolf in \cite{OW}.

\bigskip

\noindent
{\it Acknowledgments}. We are grateful to the anonymous referees for useful comments and suggestions. Parts of this work were done during the visit of the first author in SUSTech. He would like to thank Prof. Wang Xiaoming for his hospitality.
Z.H. is partially supported by the GRF grant no. 17308818 and no. 17308017 of University Grants Committee of Hong Kong SAR, China. A.P. is partially supported by the NSF grant DMS-2001224, 
and within the framework of the HSE University Basic Research Program and by the Russian Academic Excellence Project `5-100'.

\section{Feigin-Odesskii brackets}\label{FO-br-sec}

In this section we discuss some aspects of the Poisson brackets $q_{n,k}(C)$ on projective spaces defined by Feigin-Odesskii.
We use the modular definition of these brackets obtained by studying vector bundle extensions of a fixed stable vector bundle $\xi$ on $C$
by $\OO_C$.

\subsection{Formula for the Poisson bracket as a Massey product}

We start by giving the definition of the Feigin-Odesskii bracket $q_{n,k}(C)$ on the projective space
$\PP\Ext^1(\xi,\OO)=\PP H^1(C,\xi^\vee)$ following \cite[Sec.\ 5.2]{HP1}. 
Let $\xi$ be a stable vector bundle on an elliptic curve $C$ of degree $n>0$ and rank $k$. Let us fix a trivialization $\om_C\simeq \OO_C$.
The construction will depend on $\xi$ and a trivialization of $\om_C$, however, up to an isomorphism and rescaling, the bracket depends
only on $n,k$ and $C$.

Given a nonzero $\phi\in H^1(C,\xi^\vee)\simeq H^0(C,\xi)^*$, the tangent space to the projective space
is given by $H^1(C,\xi^\vee)/\lan\phi\ran$, while the cotangent space is 
$$\lan\phi\ran^\perp:=\ker\Big(H^0(C,\xi)\rTo{\phi}H^1(C,\cO_C)\Big).$$
Let 
$$0\to \OO_C\to E\to \xi\to 0$$
be the extension corresponding to $\phi$.
Let $\und{\End}(E,\OO_C)$ be the bundle of endomorphisms of $E$ preserving $\OO_C$.
It sits in a natural exact sequence
$$0\to \und{\End}(E,\OO_C)\to \und{\End}(E)\to \xi\to 0,$$
so by applying the functor $R\Hom(?,\OO_C)$, we get a boundary homomorphism
$$\de:\Hom(\und{\End}(E,\OO_C),\OO_C)\to\Ext^1(\xi,\OO_C)=H^1(C,\xi^\vee).$$
On the other hand, the exact sequence
$$0\to \xi^\vee \to \und{\End}(E,\OO_C)\to \und{\End}(\xi)\oplus\OO_C\to 0$$
induces a surjection $\Hom(\und{\End}(E,\OO_C),\OO_C)\to \lan\phi\ran^\perp\sub H^0(C,\xi)$.
The Poisson bivector $\Pi$ of the Feigin-Odesskii bracket is uniquely determined by the condition that
its value $\Pi_\phi$ at $\phi$ fits into a commutative diagram
\begin{equation}\label{Pi-definition-diagram}
\xymatrix{
H^0(C,\und{\End}(E,\OO_C)^\vee)\ar[r]^{\de}\ar[d]_{\alpha_1} &H^1(C,\xi^\vee)\ar[d]_{\alpha_2}\\
\lan\phi\ran^\perp\ar[r]^{\Pi_\phi} &H^1(C,\xi^\vee)/\lan\phi\ran
}
\end{equation}

We are going to show that this Poisson bracket can be computed as a triple Massey product. 
We refer to \cite[Sec.\ 2]{FP} for a general background on Massey products. What is important for us is that they can be calculated
in two ways, either using the triangulated structure (this definition has its origin in Toda brackets, see \cite{Cohen}), or
using the dg-resolutions.

\begin{lemma}\label{Poisson-Massey-lem} 
The Poisson bracket $\Pi_\phi:\lan\phi\ran^{\perp}\to H^1(\xi^\vee)/\lan\phi\ran$ is given by $x\mapsto MP(\phi,x,\phi)$, where
we use the triple Massey product
$$\xi[-1]\rTo{\phi}\OO\rTo{x} \xi\rTo{\phi}\OO[1].$$
Equivalently, for $s_1,s_2\in \lan\phi\ran^{\perp}\sub H^0(C,\xi)$ one has
$$\Pi_\phi(s_1\we s_2)=\pm\lan \phi, MP(s_1,\phi,s_2)\ran.$$
\end{lemma}

\begin{proof} One way to get the first formula is to use the formula for $\Pi_\phi$ in terms of Cech resolutions given in \cite[Sec.\ 5.2]{HP1}.
We will instead use the standard recipe for calculating triple Massey products based on including the first arrow $\xi[-1]\to \OO$ into an exact triangle
with $E$ as the cone (see \cite[Sec.\ 2]{Cohen}. Namely, this recipe tells that the map $x\mapsto MP(\phi,x,\phi)$ fits into a commutative diagram
\begin{equation}\label{MP-diagram}
\xymatrix{
H^0(C,E^\vee\ot\xi)\ar[r]^{\de'}\ar[d]_{\beta_1} &H^1(C,E^\vee)\\
\lan\phi\ran^\perp\ar[r]^{MP(\phi,?,\phi)} &H^1(C,\xi^\vee)/\lan\phi\ran\ar[u]_{\beta_2}
}
\end{equation}
where $\de'$ is the boundary homomorphism obtained by applying $R\Hom(E,?)$ to the extension sequence.
Now the assertion follows easily from the commutative diagram
\[\xymatrix{
H^0(C,\und{\End}(E,\OO_C)^\vee)\ar[r]^{\de}\ar[d]_{\gamma_1} &H^1(C,\xi^\vee)\ar[d]_{\gamma_2}\\
H^0(C,E^\vee\ot\xi)\ar[r]^{\de'} &H^1(C,E^\vee)\\
}\]
together with the fact that the vertical arrows in \eqref{Pi-definition-diagram} and \eqref{MP-diagram}
are related by $\alpha_1=\beta_1\gamma_1$, $\gamma_2=\beta_2\alpha_2$.

Next, we note that in terms of $A_\infty$-structure obtained by homological perturbation
we have
$$MP(\phi,s_1,\phi)\equiv \pm m_3(\phi,s_1,\phi) \mod \lan\phi\ran.$$
Next, we use the cyclic symmetry (here we use the existence of a cyclic minimal $A_\infty$-structure that follows from \cite[Sec.\ 6.5]{Polcyc}):
$$\lan m_3(\phi,s_1,\phi),s_2\ran=\pm \lan \phi,m_3(s_1,\phi,s_2)\ran.$$

It remains observe that in the right-hand side of the last formula we can replace $m_3(s_1,\phi,s_2)$ by the corresponding Massey product 
$$MP(s_1,\phi,s_2)\in H^0(C,\xi)/\lan s_1,s_2\ran.$$
Indeed, the pairing with $\phi$ is zero on the subspace $\lan s_1,s_2\ran\sub\lan\phi\ran^{\perp}$.
\end{proof}

\begin{remark} The sign ambiguity in Lemma \ref{Poisson-Massey-lem} (and in other statements below involving Massey products)
can be resolved: the signs appear from the cyclicity constraint for $A_\infty$-structures and from relating Massey products with $m_3$
(see e.g., \cite[Sec.\ 2]{FP}). For our purposes the exact value of the sign is not important.
\end{remark}

We have the following nice formula for the rank of the Poisson bracket $\Pi$ on $\PP\Ext^1(\xi,\OO)$.

\begin{prop}\label{Poisson-rank-prop} 
For a non-trivial extension
$$0\to \OO_C\to E\to \xi\to 0$$
with the class $\phi\in \Ext^1(\xi,\OO_C)$, one has
$$\rk \Pi_\phi=\deg(\xi)-\dim\Hom(E,E).$$
\end{prop}

\begin{proof}
By definition, the map $\Pi_{\phi}:\lan\phi\ran^\perp\to H^1(\xi^\vee)/\lan\phi\ran$
fits into the following sequence of arrows, whose composition is the cup product with $\phi$:
$$\Hom(E,\xi)\to\lan\phi\ran^\perp\rTo{\Pi_{\phi}} H^1(\xi^\vee)/\lan\phi\ran\to
H^1(E^\vee)=\Ext^1(E,\OO),$$
where the first map is a surjection induced by the natural map $\Hom(E,\xi)\to \Hom(\OO_C,\xi)$ and the last map is an
injection induced by the natural map $H^1(\xi^\vee)\to H^1(E^\vee)$.
Hence, the rank of $\Pi_{\phi}$ is equal to the rank of the cup product with $\phi$ map,
$$\Hom(E,\xi)\rTo{\phi} \Ext^1(E,\OO).$$ 
Note that $\Hom(\xi,\OO_C)=0$ since $\xi$ is stable of positive slope, and hence, $\Hom(E,\OO_C)=0$ since the extension does not split.
Hence, the kernel of the above map is exactly $\Hom(E,E)$. Furthermore, we have $\Ext^1(E,\xi)=\Hom(\xi,E)^\vee=0$ since the extension does not split.
Hence, by the Riemann-Roch formula $\dim \Hom(E,\xi)=\deg(\xi)$ and the assertion follows.
\end{proof}



\subsection{Recovering the elliptic curve from the Poisson bracket}

\begin{theorem}\label{reconstr-thm}
Fix an integer $d>2$. 
Suppose $\xi$ is a stable vector bundle of rank $r<d-1$ and degree $d$ on an elliptic curve $C$, and $\xi'$ a stable vector bundle of rank $r'<d-1$ and the same degree $d$ on 
another elliptic curve $C'$.
If there exists a Poisson isomorphism $\PP \Ext^1(\xi,\OO_C)\simeq \PP \Ext^1(\xi',\OO_{C'})$ then $C\simeq C'$.
\end{theorem}

A trivial example is when $r=1$ and $d=3$: then $C$ is recovered as the vanishing locus of the Poisson structure. If $r=1$ and $d=4$, then $C$ is a connected component
of the vanishing locus of the Poisson structure (the entire vanishing locus is the union of $C$ with $4$ points).

The proof is based on the following observation. We fix an elliptic curve $C$ and a stable vector bundle $\xi$ as in the above Theorem.
Let $c=gcd(d,r+1)$. 

\begin{prop}\label{Poisson-degeneration-prop}
The generic rank of the Poisson structure $\Pi$ on $\PP \Ext^1(\xi,\OO_C)$ is $d-c$.
Let $Z\sub \PP  \Ext^1(\xi,\OO_C)$ be the Zariski closure of the set of all points where the rank of $\Pi$ is $d-c-2$.
Then each nonrational irreducible component of $Z$ is birational to $\AA^m\times C$ for some $m$, and there exists at least one such component.
\end{prop}

Let us set 
$$v_0:=(\frac{d}{c},\frac{r+1}{c})\in \ZZ^2,$$ 
and let $\mu_0=d/(r+1)$ be the corresponding slope.
Let also set
$$v:=(d,r).$$
We denote by $\chi:\ZZ^2\times\ZZ^2\to \ZZ$ the bilinear form
$$\chi((d_1,r_1),(d_2,r_2))=d_2r_1-d_1r_2.$$
For a vector bundle $E$ we denote by $v(E)$ the corresponding vector $(\deg(E),\rk(E))$.

\begin{lemma}\label{rank-classification-lem}
Let $E_\phi$ denote the extension corresponding to a nonzero class $\phi\in \Ext^1(\xi,\OO_C)$.

\noindent
(i) We have $\rk \Pi_\phi\le d-c$ with equality if and only if $E_\phi=\bigoplus E_i$ where $E_i$ indecomposable bundles of slope $\mu_0$ with
$\Hom(E_i,E_j)=0$ for $i\neq j$.

\noindent
(ii) One has $\rk \Pi_\phi=d-c-2$ in one of the two cases:
\begin{itemize}
\item $E_\phi\simeq E_1\oplus E_2$, where both $E_1$ and $E_2$ are stable, $\chi(v(E_1),v_0)=1$ (and hence, $\chi(v(E_2),v_0)=-1$);
\item $E_\phi\simeq E_1\oplus E_2\oplus\ldots\oplus E_m$, where all $E_i$ are indecomposable of slope $\mu_0$, $E_1$ is stable and 
$\Hom(E_i,E_j)=0$ for $i\neq j$, $(i,j)\neq (1,2),(2,1)$.
\end{itemize}
The second case occurs only for $c>1$.
\end{lemma}

\begin{proof}
(i),(ii) By Proposition \ref{Poisson-rank-prop}, we have to prove that $\dim \End(E_\phi)\ge c$ and to study the cases where
we have an equality and the cases where $\dim \End(E_\phi)=c+2$. 

Assume first that $E_\phi$ is indecomposable (and hence, semistable).
The abelian category of semistable bundles $SB(\mu_0)$ of slope $\mu_0$ is equivalent to the category of torsion sheaves in such a way
that stable bundles of slope $\mu_0$ are simple objects in $SB(\mu_0)$. Hence, a semistable bundle $F$ of slope $\mu_0$
has length $\ell$ in this category if and only if $v(F)=\ell\cdot v_0$. Since $v(E_\phi)=cv_0$, if $E_\phi$ is indecomposable it has
$\dim\End(E_\phi)=c$.

Now let 
$$E_\phi=E_1\oplus\ldots\oplus E_m,$$ 
where $m\ge 2$, each $E_i$ is indecomposable, and $\mu(E_1)\le \mu(E_2)\le\ldots\le \mu(E_m)$. 
Assume first that $E_\phi$ is semistable, and let $\ell_i$ be the length of $E_i$ in $SB(\mu_0)$.
We have
$$\sum_i \dim \End(E_i)=\sum_i \ell_i=\ell(E_\phi)=c.$$
Thus, $\dim \End(E_\phi)\ge c$ with equality precisely when $\Hom(E_i,E_j)=0$ for $i\neq j$.

Furthermore, if $\Hom(E_i,E_j)\neq 0$ for some $i\neq j$ then
$$\dim \Hom(E_i,E_j)=\dim \Hom(E_i,E_j)=\min(\ell_i,\ell_j).$$
Hence, if $\dim \End(E_\phi)=c+2$ then we can have at most one such pair and we should have $\min(\ell_i,\ell_j)=1$.

Next, let us consider the case when $E_\phi$ is unstable.
Then there exists $i>1$ such that $\chi(v(E_1),v(E_i))>0$. Hence, $\chi(v(E_1),v(E))>0$. It follows
$$\dim\Hom(E_1,E_2\oplus\ldots E_m)\ge \chi(v(E_1),v(E_2)+\ldots+v(E_m))=\chi(v(E_1),cv_0)=c\chi(v(E_1),v_0)\ge c.$$
Therefore,
$$\dim \End(E_\phi)\ge \sum_{i=1}^m \dim\End(E_i)+\dim\Hom(E_1,E_2\oplus\ldots E_m)\ge 2+c.$$
Furthermore, the equality is possible only if $m=2$, both $E_1$ and $E_2$ are stable and $\chi(v(E_1),v_0)=1$.
\end{proof}

\begin{remark} The proof of Lemma \ref{rank-classification-lem}(i) also shows that in the case $r=d-1$ the Feigin-Odesskii bracket is identically zero.
\end{remark}

\begin{proof}[Proof of Proposition \ref{Poisson-degeneration-prop}]
It suffices to find a finite nonempty collection of irreducible closed subvarieties, $Z_1,\ldots,Z_n$, each birational to the product of $C$ with an affine space,
such that $\cup Z_i$ contains every point with $\rk \Pi_\phi=d-c-2$ and also at a generic point of each $Z_i$ we have $\rk \Pi_\phi=d-c-2$.

\medskip

\noindent
{\bf Step 1.}
First, let us fix a decomposition $cv_0=v_1+v_2$ in $\ZZ^2$, with $v_i=(d_i,r_i)$ and $r_i>0$, such that 
$$\chi(v_1,v_0)=1 \ \text{ and } \chi(v_2,v)=d_2-c>0.$$
For each such decomposition we will construct an irreducible subvariety $Z(v_2)$ in $\PP \Ext^1(\xi,\OO_C)$, which contains
all $\phi$ with $E_\phi\simeq \xi_1\oplus \xi_2$, where $\xi_1$ and $\xi_2$ are stable with $v(\xi_1)=v_1$, $v(\xi_2)=v_2$
(i.e., all points $\phi$ of the first type from Lemma \ref{rank-classification-lem}(ii)). Furthermore, we will check that a generic point of $Z(v_2)$
is a point of this type.

Let $\MM(v_2)$ denote the moduli space of stable bundles $F$ with $v(F)=v_2$ (note that $\MM(v_2)\simeq C$).
Let us consider the projective bundle $X\to \MM(v_2)$ with fiber over $\xi_2$ given by $\PP\Hom(\xi_2,\xi)$,
and let $X_0\sub X$ be the open subset corresponding to injective morphisms $\xi_2\to \xi$.
Over $X_0$ we have a projective bundle $Y\to X_0$ associated with the vector bundle
with fibers 
$$\ker(\Ext^1(\xi,\OO_C)\to \Ext^1(\xi_2,\OO_C))$$
(here we use the fact that this map of $\Ext^1$'s is surjective).
Note that $X_0$ and $Y$ are irreducible and $Y$ is birational to the product of $\MM(v_2)$ with an affine
space of dimension
$$\chi(v_2,v)-1+d-d_2-1=d-c-2.$$
We have an obvious morphism $Y\to \PP\Ext^1(\xi,\OO_C)$ and we denote by $Z(v_2)$ the closure of its image.
It is clear that the image of $Y$ consists of all $\phi$ which split over some embedding $\xi_2\to \xi$, with
$v(\xi_2)=v_2$.

Assume that $E_\phi\simeq \xi_1\oplus \xi_2$ where $\xi_i$ are stable and $v(\xi_i)=v_i$.
Then both components of the embedding $\OO_C\to \xi_1\oplus \xi_2$ are nonzero (otherwise the quotient
would be decomposable). Hence, the intersection of the image of $\OO_C$ with $0\oplus \xi_2$ is zero, which implies
that the composed map $\xi_2\to E_\phi\to \xi$ is an embedding. Since $\phi$ splits over $\xi_2\to \xi$,
we see that $\phi$ is contained in $Z(v_2)$.

For a generic point of $X_0$ the quotient $\xi/\xi_2$ will be semistable, for a generic point
of $Y$ the corresponding extension of $\xi/\xi_2$ by $\OO_C$ will be semistable with the vector $v_1$,
hence, stable. So the corresponding $E_\phi$ sits in an exact sequence
$$0\to \xi_2\to E_\phi\to \xi_1\to 0$$
which necessarily splits since $\mu(\xi_2)>\mu(\xi_1)$.

Now let us check that the map $Y\to Z(v)$ is birational. It is enough to check that if $E_\phi\simeq \xi_1\oplus \xi_2$ for some $\phi$, with $\xi_i$ as above, then there
is a unique $\xi'_2\in \MM(v_2)$ and a unique nonzero morphism $\xi'_2\to \xi$, up to rescaling, such that $\phi$ splits over this morphism.
But $\Hom(\xi'_2,\xi_1)=0$ and $\Hom(\xi'_2,\xi_2)\neq 0$ only when $\xi'_2=\xi_2$.
Furthermore, if $\phi$ splits over a morphism $\xi_2\to \xi$ then this morphism factors through $E_\phi$ and the statement follows from the fact
that $\dim \Hom(\xi_2,E_\phi)=1$.

\medskip

\noindent
{\bf Step 2.}
We claim that there exists at least one decomposition $cv_0=v_1+v_2$ as in Step 1. Indeed, 
assume first that $(r+1)/c>1$. Since $d/c$ and $(r+1)/c$ are relatively prime,
there exists a unique pair of integers $(r_1,d_1)$ with $0\le r_1<(r+1)/c$ such that
$$\frac{d}{c}\cdot r_1=d_1\cdot \frac{r+1}{c}+1.$$ 
Furthermore, we necessarily have $r_1>0$. We define $v_2$ as $cv_0-v_1$. 

Note that since $d>r+1$, $d$ cannot divide $r+1$, so $\frac{d}{c}>1$. In particular, we cannot have
$v_2=v$, so it enough to check the non-strict inequality $\chi(v_2,v)\ge 0$, which is equivalent to
$$d_1\le d-c.$$ 
The inequality $r_1<(r+1)/c$ implies that
$d_1<\frac{d}{c}$, which gives the required inequality for $c=1$. For $c\ge 2$ we use in addition
$$\frac{d}{c}\le 2(\frac{d}{c}-1)\le c(\frac{d}{c}-1)=d-c.$$ 

In the remaining case $c=r+1$ we can just take $r_1=1$ and $d_1=\frac{d}{c}-1$.

\medskip

\noindent
{\bf Step 3.}
Assume that $c>2$. We will construct an irreducible subvariety $Z_0$ in $\PP \Ext^1(\xi,\OO_C)$ which contains
all the points $\phi$ of the second type described in Lemma \ref{rank-classification-lem}(ii). Furthermore, 
we will check that a generic point of $Z_0$ is a point of this type. 

First, we observe that for every point $\phi$ of the second type from Lemma \ref{rank-classification-lem}(ii),
there exists an embedding $E_1^{\oplus 2}\to E_\phi$ such that the quotient is semistable (of slope $\mu_0$).
Indeed, since $\Hom(E_1,E_2)\neq 0$, there exists an embedding $E_1\to E_2$ with the semistable quotient,
and the assertion follows.

Now let $X$ be the relative Grassmannians of $2$-planes in the bundle over $\MM(v_0)$ with the fiber $\Hom(\xi_0,\xi)$ over $\xi_0\in \MM(v_0)$.
Let us denote by 
$X_0\sub X$ the open subset consisting of $2$-planes $P\sub \Hom(\xi_0,\xi)$ such that the corresponding map $P\ot \xi_0\to \xi$ is
injective. Let $Y\to X_0$ denote the projectivization of the vector bundle
with fibers 
$$\ker(\Ext^1(\xi,\OO_C)\to \Ext^1(P\ot \xi_0,\OO_C)).$$
We have an obvious morphism $Y\to \PP\Ext^1(\xi,\OO_C)$ and we denote by $Z_0$ the closure of its image.

Assume that $E_\phi$ is of the second type from Lemma \ref{rank-classification-lem}(ii). Then we have an embedding $\xi_0^{\oplus 2}\to E_\phi$
such that the quotient is a nonzero semistable bundle $E'$ of slope $\mu_0$ (here we use the assumption $c>2$). We claim that the
composed map $\OO_C\to E'$ is nonzero. Indeed, otherwise we would have a nonzero map from $E_\phi/\OO_C\simeq \xi$
to $E'$ which is impossible since $\mu(E')=\mu_0<\mu(\xi)$. Thus, the composed map
$$\xi_0^{\oplus 2}\to E_\phi\to \xi$$
is injective, and we see that $\phi$ lies in the image of $Y$.

We claim that for a generic point of $Y$ the quotient $\xi/(P\ot \xi_0)$ is semistable and the corresponding extension $E'$ of $\xi/(P\ot \xi_0)$ by $\OO_C$ is also 
semistable. Hence, we get an exact sequence
$$0\to \xi_0^{\oplus 2}\to E_\phi\to E'\to 0$$
with $E'$ semistable of slope $\mu_0$. Furthermore, for a generic point we will have $\Hom(\xi_0,E')=0$ and $\dim\End(E')=c-2$, so the sequence will split
and $E_\phi$ will be of the second type from Lemma \ref{rank-classification-lem}(ii).

To see that the map $Y\to Z_0$ is birational, we first observe that if $\phi$ is such that $E_\phi=\xi_0^{\oplus 2}\oplus E'$
is of type from Lemma \ref{rank-classification-lem}(ii), then for any stable $\xi'_0$ of slope $\mu_0$ one has $\dim\Hom(\xi'_0,E_\phi)\le 1$ unless
$\xi'_0\simeq \xi_0$. Furthermore, the $2$-dimensional subspace of $\Hom(\xi_0,\xi)$ is recovered from $E_\phi$ as the image of the embedding
\begin{equation}\label{Hom-xi0-Ephi-map}
\Hom(\xi_0,E_\phi)\to \Hom(\xi_0,\xi).
\end{equation}

It is also easy to see that $Y$ is birational to $\AA^{d-5}\times C$.

\medskip

\noindent
{\bf Step 4}. Finally let us consider the case $c=2$. In this case for each of the $4$ nonisomorphic stable bundle $\xi_0$ with $v(\xi_0)=\mu_0$ 
such that $\det(\xi_0)^{\ot 2}\simeq \det(\xi)$, we define a rational subvariety $Z(\xi_0)\sub \PP \Ext^1(\xi,\OO_C)$ as follows.

Let $X(\xi_0)$ denote the Grassmannian of $2$-planes in $\Hom(\xi_0,\xi)$ and let $X_0(\xi_0)\sub X(\xi_0)$ be the open subset
consisting of $P$ such that the corresponding map $P\ot \xi_0\to \xi$ is surjective. In this case the kernel is necessarily isomorphic to $\OO$, so
we get a well defined map $X_0(\xi_0)\to \PP \Ext^1(\xi,\OO_C)$. We let $Z(\xi_0)$ be the closure of its image.

It is clear that the image of $X_0(\xi_0)$ consists precisely of points $\phi$ such that $E_\phi\simeq \xi_0^{\oplus 2}$.
As in Step 3, the point of the Grassmannian is recovered from $E_\phi$ as the image of the map \eqref{Hom-xi0-Ephi-map}.
\end{proof}

\begin{proof}[Proof of Theorem \ref{reconstr-thm}]
By Proposition \ref{Poisson-degeneration-prop}, the isomorphism class of a variety $Z$, and hence a birational class of $\AA^m\times C$ is determined by the
Poisson structure. Namely, $Z$ is the closure of the set of points where the rank of the Poisson structure drops by $2$ compare to the generic rank.
But it is well known that $\AA^m\times C$ and $\AA^n\times C'$ can be birational only if $C\simeq C'$.
\end{proof}

\section{Shifted Poisson moduli stacks with singular source}\label{singular-source-sec}

Throughout this section we fix a base commutative Noetherian ring $k$ of residue characteristic 0. All stacks and schemes are over $k$ unless we specify otherwise. 
We call a $k$-scheme $X$ flat, proper or projective if the 
structure morphism $X\to \Spec k$ is such.

For the basics on derived symplectic and Poisson geometry, we refer to Section 1 of \cite{PTVV} and Section 2, 3 of \cite{HP1}.

\subsection{$\OO$-orientations}

Let us recall one of the main results in \cite{PTVV}. 

\begin{theorem}(Theorem 2.5 \cite{PTVV})\label{transgression}
Let $F$ be a locally geometric derived stack locally of finite presentation over $k$ equipped with an $n$-shifted symplectic form $\omega$. Let $X$ be an $\cO$-compact derived
stack over $k$ equipped with an $\cO$-orientation $[X]: C(X,\cO_X)\to k[-d]$ of degree $d$. Assume that the derived mapping stack
$\Map(X,F)$ is itself locally geometric and locally of finite presentation over $k$. Then $\Map(X,F)$ carries a canonical $(n-d)$-shifted symplectic 
structure.
\end{theorem}
The definition of being $\cO$-compact can be found in Definition 2.1 of \cite{PTVV}. Any quasi-projective scheme is $\cO$-compact. By definition
$C(X,\cO_X)$ is defined to be $\RHom(\cO_X,\cO_X)$, which can be represented by the Cech complex computing cohomology of $\cO_X$.

\begin{definition}\label{O-ori}
Let $X$ be an $\cO$-compact derived stack and $d\in\ZZ$. An {\it $\cO$-orientation of degree $d$} on $X$ consist of a morphism of complexes
\[
[X]: C(X,\cO_X)\to k[-d],
\] such that for any $A\in \cdga_k^{\leq 0}$ and any perfect complexes $E$ on $X_A:=X\times \Spec A$, the morphism 
\[
C(X_A,E)\to C(X_A,E^\vee)^\vee[-d]
\] induced by 
\[
[X_A]:=[X]\ot id: C(X_A,\cO_{X_A})\simeq C(X,\cO_X)\ot_k  A\to A[-d]
\]
is a quasi-isomorphism of $A$-dg-modules.
\end{definition}
\begin{definition}
Let  $X$ be a projective scheme over $\Spec k$. We call $X$ is \emph{Gorenstein Calabi-Yau $d$-fold} if 
\begin{enumerate}
\item[$(1)$] the dualizing complex $\omega_X$ is invertible;
\item[$(2)$] there is an isomorphism $\cO_X\cong \omega_X$;
\item[$(3)$] $X$ is connected.
\end{enumerate}
\end{definition}

\begin{lemma}\label{Gor-O-abs}
Let $X$ be a Gorenstein Calabi-Yau $d$-fold. Then $X$ admits an $\cO$-orientation.
\end{lemma}
\begin{proof}
Because $X$ is Gorenstein, the dualizing complex $\omega_X$ is quasi-isomorphic to an invertible sheaf. A Calabi-Yau structure corresponds to a trivialization $\eta: \cO_X\cong \omega_X$. Let $\cE$ be a perfect complex on $X$. Denote $(\cA_\bullet,d)$ for total complex of the sheaf endomorphism complex $\underline{\Hom}(\cE,\cE)$. Then $\cA_0=\bigoplus_i \ul{\Hom}(\cE^i,\cE^i)$. Denote by
\[
\tau^\p: \cA_0\to \cO_X
\] the (super)trace morphism. We extend $\tau^\p$ to a morphism from $\cA_\bullet$ to $\cO_X$ by pre-compose it with the natural projection. Define $\tau$ to be the composition $\eta\circ \tau^\p$. Clearly, $\tau^\p\circ d=0$.
The canonical trace morphism $H^d(\omega_X)\to k$ (from the definition of dualizing complex), together with the CY structure $\eta$, defines the desired morphism 
\[
[X]: C(X,\cO_X)\to k[-d].
\] 

Now we consider the case when the base is an affine derived scheme. Given $A\in \cdga_k^{\geq 0}$, denote by $X_A$ the product $X\times_k \Spec A$. Let $\cE$ be a perfect complex on $X_A$. We have a cartesian diagram of derived schemes
\[
\xymatrix{
X_{H^0A}\ar[d]^v \ar[r]^j & X_A\ar[d]^u\\
\Spec H^0A\ar[r]^i  & \Spec A
}
\]
where $X_{H^0A}:=X\times_k \Spec H^0(A)$.
By the base change formula of derived schemes (Prop 1.4 \cite{Toenbasechange}), there is an equivalence
\[
i^*u_*\cF\simeq v_*j^*\cF
\] for any quasi-coherent complex $\cF$ on $X_A$. All functors are derived. Take $\cF=\cE\ot \cE^\vee$. We need to check that the morphism
\[
\eta_\cE: u_*\cE\to \RHom_A(u_*(\cE^\vee), A[-d])
\] is an isomorphism in $\D(A)$. We claim that it is equivalent to show that 
\[
i^*(\eta_\cE): i^*u_*\cE\to i^*\RHom(u_*(\cE^\vee), A[-d])
\] is an isomorphism in $\D(H^0A)$. Because $u$ is proper and flat, both $u_*\cE$ and $\RHom(u_*(\cE^\vee), A[-d])$ are perfect $A$-modules. It suffices to show a perfect $A$-module $M$ is acyclic if and only if $i^*M$ is acyclic. Because $M$ is perfect and $A$ is nonpositively graded, there exists $n$ such that $H^i(M)=0$ for $i>n$. By spectral sequence, 
\[
H^n(M)=H^n(M\ot_A H^0A)=H^n(i^*M)=0.
\] By induction, $M$ is acyclic. The claim is proved.

By base change, $i^*(\eta_\cE)$ is isomorphic to the morphism
\[
\eta_{j^*\cE}: v_*j^*\cE\to \RHom_{H^0A}(v_*(j^*\cE^\vee),H^0A[-d]),
\] induced by the bilinear map
\[
j^*\cE\ot j^*(\cE^\vee)\to \omega_v\cong\cO_{X_{H^0A}}.
\]
Then $\eta_{j^*\cE}$ is an isomorphism in $\D(H^0A)$ by Grothendieck duality for the scheme morphism $v$.
\end{proof}

\subsection{Shifted Poisson structure on the moduli of complexes}
We briefly recall the construction of moduli stack of complexes following \cite[Section 2]{HP-Bos}. The basics on graded mixed objects can be found in \cite[Section 1]{CPTVV}.  Those readers who are familar with \cite{CPTVV} can skip the first two pages and read Theorem \ref{Rperfalgebraic} directly. 

Let $k$ be a Noetherian commutative ring. Let $C(k)$ be the category of unbounded dg-$k$-modules with the standard model structure, where weak equivalences are quasi-isomorphisms and fibrations are epimorphisms of cochain complexes. Let $M$ be a symmetric monoidal model category with a $C(k)$-enrichment.

A \emph{graded mixed object} in category $M$ is  a $\ZZ$-family of $\{E(p)\}_{p\in\ZZ}$ of objects in $M$ together with morphisms in $M$
\[
\ep=\{\ep_p: E(p)\to E(p+1)[1]\}_{p\in\ZZ}
\] where $[1]$ is the shift functor defined by the $C(k)$-enrichment, and $\ep^2=0$. We write $(E,\ep)$ for the family together with the differential. A morphism 
\[
f: (E,\ep)\to (F,\ep)
\] is a family of maps $\{f(p): E(p)\to F(p)\}_{p\in\ZZ}$ in $M$ that commutes with $\ep$. We call a graded mixed object in $M$ \emph{bounded} if $E(p)=0$ except for finitely many $p$. Denote the category of graded mixed objects in $M$ by $\ep-M^\gr$. 

The category $M^\gr:=\prod_{p\in\ZZ} M$ is naturally a symmetric monoidal model category enriched in $C(k)$, inherited from $M$. There is a forgetful functor
\[
\ep-M^\gr\to M^\gr
\] forgetting the $k[\ep]$-structure. Equip $\ep-M^\gr$ with the symmetric monoidal model structure through the forgetful functor. Given a triangulated dg category $T$, following \cite{TV07} we denote the category of perfect (or compact) objects by $T_{pe}$. Suppose  $M$ is triangulated and admits arbitrary coproduct. An object of $\ep-M^\gr$ is called \emph{perfect} if it is a compact object in $M^\gr$. Denote by $\ep_{pe}-M^\gr$ the subcategory of $\ep-M^\gr$ consisting of perfect objects.

Let $E, F$ be two mixed graded objects.
We define the external hom by
\[
\Hom^{\NN}_\ep(E,F):=\prod_{p\in\ZZ}\Big(\Hom^{\NN}_\ep(E,F)(p)\Big)
\] where 
\[
\Hom^{\NN}_\ep(E,F)(p)=\prod_{q\in\NN}\Hom_{M}(E(q),F(q+p)).
\] The differential 
\[
\ep(p): \Hom^{\NN}_\ep(E,F)(p)\to \Hom^{\NN}_\ep(E,F)(p+1)[1]
\] is defined by the adjoint action of $\ep$ on $E$ and $F$. This defines a $C(k)$-enrichment of $\ep-M^\gr$ and the forgetful functor $\ep-M^\gr\to M^\gr$ is $C(k)$-enriched.

\begin{example}\label{pq}
When $M=C(k)$, denote the stack of perfect objects in $C(k)$ by $\dPer$, the stack of objects in $\ep_{pe}-C(k)^\gr$ by $\dePer$ the stack of perfect objects in $C(k)^\gr$ by $\dPer^\ZZ_b$. The lower index $b$ stands for bounded. We have stack morphism
\begin{equation}\label{CD:1}
\xymatrix{
 &\dePer\ar[rd]^q\ar[ld]_p\\
 \dPer^\ZZ_b  & & \dPer
}
\end{equation} where $p$ is induced by the forgetful functor and $q$ is induced by the functor taking the total complex.
Denote by $\bCx$ the subcategory of $\ep_{pe}-C(k)^\gr$ where $E(p)$ has perfect amplitude $[0,0]$ for all $p\in\ZZ$, and by $\dCx$ the associate stack of objects. Then the above diagram restricts to 
\begin{equation}\label{CD:2}
\xymatrix{
 &\dCx\ar[rd]^q\ar[ld]_p\\
 \RR\Vt^\ZZ_b  & & \dPer
}
\end{equation} where $\Vt$ is the stack of vector bundles. In a seminal paper \cite{TV07}, Toen and Vaquie have proved that $\dPer$ is a locally geometry stack locally of finite presentation over $k$. The same holds for $\dPer^\ZZ_b$ and $\RR\Vt^\ZZ_b$.
\end{example}

Let $X$ be a flat projective $k$-scheme and $M=\Qcoh(X)$ be the category of quasi-coherent complexes on $X$. An object of $\ep_{pe}-M^\gr$ is a graded mixed complexes of quasi-coherent complexes $\{E(p)\}_{p\in\ZZ}$ where $E(p)$ is a perfect complex on $X$ for all $p \in\ZZ$ and $E(p)=0$ except for finitely many $p$. Denote by $\bCx(X)$ the subcategory of $\ep_{pe}-M^\gr$ where $E(p)$ has perfect amplitude $[0,0]$ for all $p\in\ZZ$. Objects of $\bCx(X)$ are simply bounded complexes of vector bundles (since we have assumed that $X$ is projective). 

\begin{theorem}\cite[Lemma 2.4, Proposition 7.3, Theorem 2.3]{HP-Bos}\label{Rperfalgebraic}
Let $X$ be a flat projective $k$-scheme. Denote by $\dePer(X)$ the stack of objects in $\ep_{pe}-\Qcoh(X)^\gr$ and $\dCx(X)$ the stack of objects of $\bCx(X)$. Then there is an equivalence of stacks
\[
\dePer(X)\simeq \Map\Big(X,\dePer\Big)\simeq \Map\Big(X, \Map([\AA^1\Big/\GG_m], \dPer)\Big)
\] where $\Map$ is the internal hom of the category of (derived) stacks. As a consequence, $\dePer$, $\dCx$, $\dePer(X)$ and $\dCx(X)$ are locally geometric stacks locally of finite presentation over $k$.
\end{theorem}

\begin{lemma}\label{relativesymplectic}
Let $k$ be a commutative Noetherian ring of residue characteristic 0 and $X$ be a Gorenstein Calabi-Yau $d$-fold (over $\Spec k$). Then for a given isomorphism $\cO_X\cong \omega_X$, $\dPer(X)$ admits a canonical $(2-d)$-shifted symplectic structure.
\end{lemma}
\begin{proof}
By Proposition 3.7 of \cite{TV07}, $\dPer$ is locally geometric and locally of finite presentation over $k$. It  admits a canonical 2-shifted symplectic structure by Theorem 2.12 \cite{PTVV}.  Since $X$ is projective over $k$, it is $\cO$-compact. Applying Lemma \ref{Gor-O-abs}, the isomorphism $\cO_X\cong \omega_X$ defines an $\cO$-orientation. By Lemma \ref{Rperfalgebraic}, $\dPer(X)$ is locally geometric and locally of finite presentation over $k$. Finally by Theorem \ref{transgression}, $\dPer(X)=\Map(X,\dPer)$ admits a canonical $(2-d)$-shifted symplectic structure.
\end{proof}
\begin{remark}
Since $\dVt(X)$, the stack of vector bundles on $X$, is an open substack of $\dPer(X)$, it inherits the symplectic structure on $\dPer(X)$. Since $\dPer^\ZZ_b(X)$ is locally a  finite direct product of $\dPer(X)$, therefore is also canonically symplectic. The same holds for $\RR\Vt^\ZZ_b(X)$.
\end{remark}

The following result is a version of Theorem 3.17 of \cite{HP1} for not necessarily smooth Calabi-Yau families.

\begin{theorem}\cite[Theorem 3.4]{HP-Bos}\label{existence-relative}
Let $k$ be a Noetherian commutative ring of residue charactersitic zero and
$X$ be a Gorenstein Calabi-Yau $d$-fold over $\Spec k$.
Given a trivialization $\cO_X\cong \omega_X$, the moduli stack $\dCx(X)$ admits a canonical $(1-d)$-shifted Poisson structure. 
\end{theorem}

We refer to \cite{PTVV, CPTVV} for the definitions of a shifted symplectic and a shifted Poisson structure on a derived stack. The Poisson structure in Theorem \ref{existence-relative} is indeed constructed via Lagrangian structure (see \cite{PTVV, CPTVV, HP1}) using the following result of Melani and Safronov. 

\begin{theorem}\cite[Theorem 4.22]{MS2}\label{thm:MS}
Suppose $\sX, \sY$ are locally geometric stacks locally of finite presentation.
Let $f:\sX\to \sY$ be a stack morphism. Supose that $\sY$ is equipped with an $n$-shifted symplectic form $\omega$ and $f$ is Lagrangian. Then $\sX$ is equipped with a canonical $(n-1)$-shifted Poisson structure.
\end{theorem}

\begin{proof}[Proof of Theorem \ref{existence-relative}]
By Theorem \ref{Rperfalgebraic},  the commutative diagrams \eqref{CD:1} and \eqref{CD:2} are diagrams of morphisms of locally geometric stacks locally of finite presentation, for which the notion of Lagrangian morphism is well defined.  It is proved in \cite[Theorem 3.13]{HP1} that 
\[
(p,q): \dCx : \dVt^\ZZ\times \dPer
\]
is a Lagrangian correspondence (see \cite[Appendix A]{HP-Bos} for a different proof for $\dePer$ via boundary structure, \cite[Definition 2.8]{Ca14}). By Theorem \ref{Rperfalgebraic}, Lemma \ref{Gor-O-abs} and transgression of Lagrangian structure (c.f. \cite[Theorem 2.10]{Ca14}), we produce a canonical $(1-d)$-shifted Poisson structure on $\dCx(X)$.
\end{proof}

\begin{remark}
A key feature of the Poisson structure in Theorem \ref{existence-relative} is that its weight 2 component is induced by an explicit morphism between certain complexes of coherent sheaves, whose hypercohomology cochain complexes are quasi-isomorphic to the tangent and cotangent complex of $\dCx(X)$. The formula for this morphism can be found in \cite[Theorem 4.7]{HP1} and \cite[Section 3.2]{HP-Bos}.
\end{remark}

Since in the application we need to consider relative moduli stack over a base $B$ that is not necessary affine, we make the following definition.  Let $f:X\to B$ be a scheme morphism where $B$ is a Noetherian scheme of finite type. Denote by $\dPer(X/B)$ the stack of perfect complexes on $X$ that are also $B$-perfect. Similarly, we define $\RR\Vt(X/B)$, $\dePer(X/B)$ and $\dCx(X/B)$. In this paper we only consider those $f$ that are flat and projective. In this case, we indeed have $\dPer(X/B)\simeq \dPer(X)$ and an analogue holds for $\RR\Vt(X/B)$, $\dePer(X/B)$ and $\dCx(X/B)$. However, we keep the relative notations to emphasize that we are in the relative situation.


\section{Relative Poisson structures from families of CY-curves}\label{CY-curves-sec}

\subsection{Relative Poisson structure on the relative moduli spaces of complexes}

We say that $\pi:C\to S$ is a {\it family of Gorenstein CY-curves} if $\pi$ is flat projective morphism with connected geometric fibers that are Gorenstein
of dimension $1$, such that for the relative dualizing sheaf we have
$\om_{C/S}\simeq \pi^*L_S$ for some line bundle $L_S$ on $S$.

We can consider the associated relative moduli stack of complexes $\dCx(C/S)$.
For a subset $I\sub{\mathbb Z}$, an object $F\in \perf(C)$, and a collection of vector bundles $(V_i)_{i\in I}$ on $C$,
we consider the substack $\dCx(C/S;F,(V_i)_{i\in I})$ corresponding to complexes $V_\bullet$ with fixed $i$th term given by $V_i$ for $i\in I$,
and a fixed isomorphism $V_\bullet\simeq F$ in the derived category (this substack is defined as a derived fibered product, see \cite[Cor.\ 3.20]{HP1}).

\begin{prop}\label{rel-moduli-Poisson-prop} 
Let $\cM\to S$ be an open substack in $\dCx(C/S,F,(V_i)_{i\in I})$ such that
$\cM$ admits a relative coarse moduli $\cM\to M\to S$, such that $p:M\to S$ is smooth, and $\cM\to M$ is a $\GG_m$-gerbe
(in particular $\cM$ has trivial derived structure). 
Then there exists a global section $\Pi\in \bigwedge^2 T_{M/S}\ot p^*L_S$ such that for every point $s\in S$,
the bivector $\Pi_s$ on the fiber $M_s$ is the Poisson structure induced by $0$-shifted Poisson structure on $\cM_s$.
\end{prop}

\begin{proof} First, let us consider the case when $S$ is affine.
Let $\wt{S}$ be the total space of the $\GG_m$-torsor associated with the line bundle $L_S^{-1}$, so that $S=\wt{S}/\GG_m$.
Then there is a base change diagram
\[\xymatrix{
\wt{C}\ar[r]^{\wt{p}}\ar[d]^{\wt{\pi}} & C\ar[d]^\pi \\
\wt{S}\ar[r]^p & S }
\] 
Since $\omega_{\wt{C}/\wt{S}}=\wt{p}^*\omega_{C/S}$ and $p^*L_S$ is trivial, $\wt{C}$ admits an $\cO$-orientation relative to $\wt{S}$.

Therefore, by Theorem \ref{existence-relative}, we get a $0$-shifted Poisson structure on $\cM\times_S \wt{S}$, which is a $\GG_m$-gerbe over $M\times_S \wt{S}$.
The argument of Proposition 2.6 of \cite{HP2} can be easily generalized to the relative setting. Therefore, the $0$-shifted Poisson structure on $\cM\times_S \wt{S}$ descends to a Poisson structure on $M\times_S \wt{S}$ relative to $\wt{S}$. We then obtain a global section $\Pi$ of the pull back of $\wedge^2T_M$ on $M\times_S \wt{S}$. 
It remains to prove that $\Pi$ has weight $1$ with respect to the natural action of $\GG_m$ on $\wt{S}$.

By construction, on $\wt{S}$ we have an isomorphism 
$$\th:\cO_{\wt{S}}\to p^*L_S,$$
transforming under the action of $\GG_m$ by
\begin{equation}\label{la-theta-action-eq}
\la^*\th=\la^{-1}\cdot \th.
\end{equation}
Thus, we get an induced isomorphism 
$$\th:\cO_{\wt{C}}\to \om_{\wt{C}/\wt{S}}$$
still satisfying \eqref{la-theta-action-eq}.

Recall that the tangent space to a point of $\cM_s$ is identified with the hypercohomology $\HH^1(C_s,\cC)$, where $\cC$ is some natural complex,
equipped with a chain map
\[
\partial\circ {\bf{t}}: \cC^\vee[-1] \to \cC
\] 
(see Theorem 4.7 of \cite{HP1}), so that
the bivector induced by the $0$-shifted Poisson structure is given by 
\[
\Pi:\HH^1(C_s, \cC)^\vee\simeq \HH^0(C_s,\cC^\vee\ot \om_{C_s})\to \HH^1(C_s,\cC^\vee[-1])\to \HH^1(C_s,\cC),
\] 
where the middle arrow is induced by $\theta^{-1}$ and the last map is induced by $\partial\circ{\bf{t}}$. 
It follows that
$$\la^*\Pi=\la\cdot \Pi$$
as claimed.

For not necessarily affine base $S$ we can pick an open affine covering $(S_i)$, and apply the above argument to get
sections $\Pi_i$ of $\bigwedge^2 T_{M/S}\ot p^*L_S$ over open subsets $p^{-1}(S_i)$. Furthermore, still by the affine case,
$\Pi_i$ and $\Pi_j$ have the same restrictions to every open subset of the form $p^{-1}(U)$, where $U\sub S_i\cap S_j$ is an affine open.
Hence, $(\Pi_i)$ glue into a global section of $\bigwedge^2 T_{M/S}\ot p^*L_S$.
\end{proof}

\subsection{Compatible Poisson structures from families of CY-curves}

Let $\pi:C\to S$ be a family of CY-curves, and let $L_S$ be a line bundle on $S$ such that $\om_{C/S}\simeq \pi^*L_S$.
Assume that $\cV$ a vector bundle on $C$, such that the corresponding bundles $\cV_s$ on $C_s$ are
endosimple, $R^1\pi_*\cV=0$ and 
$$\pi_*\cV\simeq V\ot \cO_S$$ 
for some vector space $V$.

Then for each $s\in S$, we have the moduli space $\cM_s$ of extensions of $\cV_s$ by $\cO_{C_s}$ on $C_s$,
which is a $\GG_m$-gerbe over 
$$M_s=\PP\Ext^1(\cV_s,\cO_{C_s})\simeq \PP H^1(C_s,\cV_s^\vee).$$
By Serre duality, we have an identification, 
$$M_s\simeq \PP H^0(C_s,\cV_s)^\vee\simeq \PP V^\vee.$$

Viewing extensions in $\MM_s$ as two-term complexes $\cO_s\to E$ with $E/\cO_s\simeq \cV_s$,
and using Proposition \ref{rel-moduli-Poisson-prop} we get a global section $\Pi$ of
the bundle $\bigwedge^2 T_{\PP V}\boxtimes L_S$ over $M=\PP V^\vee\times S$.

Note that this gives us a linear family of bivectors $\Pi_x$ on $\PP V^\vee$ parameterized by $x\in H^0(S,L_S)^\vee$.
However, we only know that $\Pi_x$ is integrable for $x$ coming from a point of $S$.


Now we specialize to the case when $S$ is a projective space, $S=\PP^N$ and $L_S=\cO_{\PP^N}(1)$.
Since in this case $S$ is identified with $\PP H^0(S,L_S)^\vee$,
the previous discussion gives the following result.

\begin{theorem}\label{proj-base-thm}
Let $\pi:C\to S=\PP^N$ be a family of Gorenstein curves of arithmetic genus 1 with $\om_{C/S}\simeq \pi^*\cO(1)$, and let
$\cV$ be a vector bundle on $C$, such that $\cV_s$ is endosimple for every $s\in S$,
$$R^1\pi_*\cV=0 \text{ and } \pi_*\cV\simeq V\ot\cO_S$$ 
for some vector space $V$. Then we get a global section $\Pi$
of $\bigwedge^2 T_{\PP V^\vee}\boxtimes \cO_{\PP^N}(1)$ over $\PP V^\vee\times \PP^N$,
such that for every $s\in \PP^N$, the bivector $\Pi_s$ defines a Poisson structure on $\PP V^\vee$.
Equivalently, we get a collection $\Pi_0,\ldots,\Pi_N$ of
Poisson structures on $\PP V^\vee$, such that $[\Pi_i,\Pi_j]=0$.
\end{theorem}

\subsection{Families of anticanonical divisors}

We will use Theorem \ref{proj-base-thm} to get compatible Poisson brackets on projective spaces
from linear systems of anticanonical divisors on surfaces.

\begin{prop}\label{surface-antican-example-prop}
(i) Let $X$ be a smooth projective surface, $W:=H^0(X,\om_X^{-1})$. Let $C\sub X\times \PP W$ be the universal
anticanonical divisor, viewed as a family over $\PP W$ via the natural projection $\pi:C\to \PP W$. 
Then $\om_{C/\PP W}\simeq \pi^*\cO(1)$. 

\noindent
(ii) In addition, let $\cV$ be a vector bundle on $X$ such that
$H^*(X,\cV\ot\om_X)=H^1(X,\cV)=0$. Then the restriction 
$$\cV_C:=\cV\boxtimes\cO|_C$$ satisfies
$R^1\pi_*\cV_C=0$, $R^0\pi_*\cV_C\simeq V\ot \cO_{\PP W}$,
where $V:=H^0(X,\cV)$.

\noindent
(iii) In the situation of (i) assume in addition that there exists a smooth anticanonical divisor $C_0\sub X$.
Then for any vector bundle $\cV$ on $X$ such that $H^*(X,\cV\ot\om_X)=0$ and the restriction $\cV|_{C_0}$ is
a semistable bundle on $C_0$ of positive degree
one has $H^1(X,\cV)=0$, i.e.,
the assumptions of (ii) are satisfied.
\end{prop}

\begin{proof} 
(i) Note that $\cO(C)\simeq \om_X^{-1}\boxtimes \cO(1)$.
Hence, by the adjunction formula we get
$$\om_{C/\PP W}\simeq (\om_X\boxtimes\cO)(C)|_C\simeq \cO\boxtimes \cO(1)|_C\simeq \pi^*\cO(1).$$

\noindent
(ii) For every anticanonical divisor $C_0\sub X$, we have a long exact sequence
\begin{align}\label{res-to-antican-divisor-seq}
&H^0(X,\cV(-C_0))\to H^0(X,\cV)\to H^0(C_0,\cV|_{C_0})\to H^1(X,\cV(-C_0))\to H^1(X,\cV)\to \nonumber\\
&H^1(C_0,\cV|_{C_0})\to H^2(X,\cV(-C_0)).
\end{align}
Now our assumptions on $\cV$ implies that $H^0(X,\cV)\to H^0(C_0,\cV|_{C_0})$ is an isomorphism and
that $H^1(C_0,\cV|_{C_0})=0$.

Finally, $R^0\pi_*\cV_C=R^0\pi_*\cV$ is trivial by base change formula.

\noindent
(iii) Let us consider the sequence \eqref{res-to-antican-divisor-seq} for a smooth anticanonical divisor $C_0$.
Since $H^*(X,\cV\ot\om_X)=0$, we deduce an isomorphism
$$H^1(X,\cV)\simeq H^1(C_0,\cV|_{C_0}).$$
But $\cV|_{C_0}$ is semistable of positive degree. It follows that 
$$H^1(C_0,\cV|_{C_0})=\Hom(\cV,\cO_{C_0})^*=0,$$ 
so $H^1(X,\cV)=0$.
\end{proof}

Now we are ready to prove our main result about families of compatible Poisson brackets coming from exceptional bundles on surfaces.

\begin{theorem}\label{exc-pair-thm}
(i) Let $X$ be a smooth projective surface with $H^{>0}(X,\OO_X)=0$ and $h^0(X,\om_X^{-1})>1$, and let $\cV$ be an exceptional vector bundle on $X$ such that $(\cO,\cV)$ is an exceptional pair and such that $c_1(\cV)\cdot c_1(\om_X^{-1})>0$.
Then there is a natural linear map
$$\kappa:H^0(X,\om_X^{-1})\to H^0(\PP H^0(X,\cV)^*,{\bigwedge}^2 T)$$
whose image consists of compatible Poisson brackets and such that
for every smooth anticanonical divisor $C\sub X$, $\kappa([C])$ is the Feigin-Odesskii bracket associated with $\cV|_C$.

\noindent
(ii) Assume in addition that $c_1(\cV)\cdot c_1(\om_X^{-1})>\rk(\cV)+1$ and that there exists a pair of non-isomorphic smooth anticanonical divisors in $X$.
Then $\ker(\kappa)$ is entirely contained in the discriminant locus (corresponding to singular anticanonical divisors).
In particular, for any smooth anticanonical divisor $C$, the Feigin-Odesskii bracket associated with $\cV|_C$ extends to a bihamiltonian structure.
If moreover every singular anticanonical divisor $C_0$ extends to a non-isotrivial pencil $\la C_0+\mu C$, with smooth $C$, then $\kappa$ is injective.
\end{theorem}

\begin{proof}
(i) It is well known that for every smooth anticanonical divisor $C\sub X$, the restriction
$\cV|_C$ is an endosimple (and hence stable) vector bundle on an elliptic curve $C_0$. This implies that the assumptions
of Proposition \ref{surface-antican-example-prop}(iii) are satisfied, and the assertion follows.

(ii) Let $[C]$ be in $\ker(\kappa)$. Assume $C$ is smooth. Pick another smooth anticanonical divisor $C'$ such that $C'\not\simeq C$. Then
$\kappa(\lan [C],[C']\ran)$ is at most $1$-dimensional, so the Feigin-Odesskii brackets associated with $\cV|_C$ and $\cV|_{C'}$ are proportional.
By Theorem \ref{reconstr-thm}, this implies that $C\simeq C'$ which is a contradiction. This shows that $\ker(\kappa)$ is contained in the discriminant locus.

Thus, for a pair $C$, $C'$ of non-isomorphic smooth anticanonical divisor on $X$, the subspace $\kappa(\lan [C],[C']\ran)$ is $2$-dimensional. Hence, we
get a bihamiltonian structure.

For the last assertion, we apply the same argument as above to a non-isotrivial pencil $\la C_0+\mu C$ with $[C_0]$ in $\ker(\kappa)$ to get a contradiction.
\end{proof}

\begin{corollary}\label{del-Pezzo-cor}
Let $C$ be a smooth cubic in $\PP^2$ and let us fix $n\le 7$. Assume that for any $n$ generic points $p_1,\ldots,p_n\in C$, there exists an
exceptional pair $(\cV,\cO)$ on the blow up $X$ of $\PP^2$ at these points, with $c_1(\cV)\cdot c_1(\om_X^{-1})>\rk(\cV)+1$.
Then the Feigin-Odesskii bracket associated with $\cV|_C$ extends to a bihamiltonian structure.
\end{corollary}

\begin{proof}
First, we pick a smooth cubic $C'\sub \PP^2$, non-isomorphic to $C$. Changing $C'$ by an auto morphism of $\PP^2$ we can assume
that $C$ and $C'$ intersect transversally. Choose $n$ points in $C\cap C'$ and consider the corresponding blow up $X$. Then
both $C$ and $C'$ lift to anticanonical divisors of $X$. Now we can apply Theorem \ref{exc-pair-thm}(ii).
\end{proof}


\begin{example} Let $X=\PP^2$ and $\cV=L=\cO(k)$, where $k=1$ or $2$. 
Then the assumptions of Theorem \ref{exc-pair-thm} are satisfied.
Note that $H^0(\PP^2,\om_{\PP^2}^{-1})$
is $10$-dimensional, while $H^0(\PP^2,L)$ is $3$-dimensional for $k=1$ and $6$-dimensional for $k=2$.
Thus, we get a set of $10$ compatible Poisson brackets
on $\PP^2$ (for $k=1$) and on $\PP^5$ (for $k=2$), containing the FO-brackets $q_{3,1}$ and $q_{6,1}$, respectively.
\end{example}

We can generalize the above example as follows (excluding the trivial cases of $q_{3,1}$, $q_{3,2}=0$ and $q_{6,5}=0$).
Let $(f_n)$ denote the Fibonacci sequence, where $f_0=0$, $f_1=1$.

\begin{prop}\label{Fibonacci-prop} 
For every $n\ge 2$, there exists a $10$-dimensional subspace of compatible Poisson brackets on $\PP^{3f_{2n-1}}$ containing 
every $q_{3f_{2n-1},f_{2n-3}}(C)$; while for $n\ge 3$, there exists a $10$-dimensional subspace of compatible Poisson brackets on $\PP^{3f_{2n-1}}$ containing 
every $q_{3f_{2n-1},3f_{2n-1}-f_{2n-3}}(C)$.
\end{prop}

\begin{proof} We apply Proposition \ref{surface-antican-example-prop} for $X=\PP^2$ by taking $\cV$ to be any exceptional bundle such that $\cV\in \langle \cO(1),\cO(2)\rangle$.
Note that the assumptions are satisfied 
The exceptional bundles we need form a helix $(E_i)$ in the category $\langle \cO(1),\cO(2)\rangle$, where $E_0=\cO(1)$, $E_1=\cO(2)$.
Then for $n\ge 0$, 
we have the following relations in the Grothendieck group 
$$[E_{-n}]=f_{2(n+1)}[E_0]-f_{2n}[E_1], \ \ [E_n]=f_{2n}[E_1]-f_{2(n-1)}[E_0].$$
Hence, for $n\ge 1$, we have
$$\rk E_{-n}=f_{2n+1}, \dim H^0(E_{-n})=3f_{2n-1}, \ \ \rk E_n=f_{2n-1}, \dim H^0(E_n)=3f_{2n+1}.$$
This leads to the linear maps from $H^0(\PP^2,\cO(3))$ to the spaces of bivectors on the claimed projective spaces whose image consist of compatible Poisson
brackets.

Finally, let us check that the linear maps 
$$H^0(\PP^2,\OO(3))\to H^0(\PP^N,{\bigwedge}^2 T_{\PP^N})$$
corresponding to our families of Poisson brackets are injective.
Since all exceptional bundles on $\PP^2$ are $\GL_3$-equivariant, 
the above map is compatible with $\GL_3$-action. Hence, the kernels of the above linear maps are $\GL_3$-subrepresentations
in $H^0(\PP^2,\OO(3))$. But the representation of $\GL_3$ on $H^0(\PP^2,\OO(3))$ is irreducible, so either the kernel is zero,
or the entire map is zero. Thus, it is enough to show that our construction does not give identically zero brackets. But this
follows from the well known fact that the Feigin-Odesskii bracket $q_{n,k}(C)$ associated with an elliptic curve $C$ is nonzero provided $n>k+1$ 
(this follows e.g., from Proposition \ref{Poisson-degeneration-prop}).
\end{proof}

\begin{example}\label{OW-ex}
Let $X=F_n=\PP(\cO\oplus \cO(n))$, the Hirzebruch surface (or $\PP^1\times \PP^1$, for $n=0$), and let $p:X\to \PP^1$ be the projection.
Then 
$$\om_X^{-1}\simeq p^*(\cO(n+2))(2),$$
so 
$$H^0(X,\om_X^{-1})\simeq H^0(\PP^1,\cO(n+2)\oplus \cO(2)\oplus \cO(-n+2)).$$
For $|n|\le 3$, this is a $9$-dimensional vector space.
We can take 
$$\cV=L:=p^*(\cO(k))(1).$$
Then $Rp_*(L\ot\om_X)=0$, so $H^*(X,L\ot\om_X)=0$. Also, $Rp_*(L)\simeq \cO(k)\oplus \cO(k-n)$,
so for $k\ge n-1$, $H^1(X,L)=0$. Thus, the conditions of Proposition \ref{surface-antican-example-prop}
are satisfied in this case, and for $|n|\le 3$, we get a family of $9$ compatible Poisson brackets on the projective space $\PP^{2k+1-n}$.
Later we will show that the cases $n=1$ and $n=2$ correspond to the examples in Odesskii-Wolf \cite{OW} (see Sec.\ \ref{OW-sec}) and that the
corresponding $9$ brackets are linearly independent.
\end{example} 

\begin{prop}\label{biham5-prop} 
For any $d>r>0$ such that $d\equiv \pm 1\mod(r)$ and $r$ is odd and any elliptic curve $C$, 
the Poisson bracket $q_{d,r}(C)$ extends to a bihamiltonian structure.
\end{prop}

\begin{proof} Let us realize $C$ is a smooth cubic in $\PP^2$ and consider the blow up $X$ of $\PP^2$ at $5$ generic points $p_0,p_1,\ldots,p_4$ on $C$
(so that no three are collinear). Then $X$ is a del Pezzo surface.
By Corollary \ref{del-Pezzo-cor}, it is enough to construct an exceptional bundle $E$ over $X$ of rank $r$ and $\chi(E)=d$ such that $(\OO_X,E)$ is an exceptional pair.

For the construction of $E$, it will be more convenient to view $X$ as the blow up of a Hirzebruch surface $F$ at $4$ points.
More precisely, we need two such realizations with $F=F_n$, where $n$ is either $1$ or $0$. 
First, we can identify the blow up of $\PP^2$ at $p_0$ with the Hirzebruch surface $F_1$
and then view $X$ as the blow up of $F_1$ at $p_1,p_2,p_3,p_4$. The second way, is to identify the blow up of $\PP^2$ at $p_0$ and $p_1$ with
the blow up of $F_0=\PP^1\times \PP^1$ at one point $p'_1$, so we can view $X$ as the blow up of $F_0$ at $p'_1,p_2,p_3,p_4$. 
We denote by $\pi:X\to F_n$ the blow down map, and by $p:F\to \PP^1$ the $\PP^1$-fibration map and by $\OO(1)$ the corresponding line bundle on $F_n$,
as in Example \ref{OW-ex}.

We observe that any $E$ in the subcategory
$${\mathcal C}:=\langle \pi^*(p^*D(\PP^1)(1)),\OO_{e_1},\ldots,\OO_{e_4}\rangle,$$
where $e_i$ are exceptional divisors for $\pi$, will have $\Hom^*(E,\OO_X)=0$. 
Let us start with an exceptional pair
$$V_1=\pi^*(p^*\OO(k-1)(1)), \ \ V_2=\pi^*(p^*\OO(k)(1))(-e_1-e_2-e_3-e_4)$$
in ${\mathcal C}$. We have $\Ext^i(V_1,V_2)=0$ for $i\neq 1$, while
$\Ext^1(V_1,V_2)$ is $2$-dimensional. We claim that this
implies that in the helix generated by $V_1$ and $V_2$ we will find (up to a shift) vector bundles $V$ with
$$[V]=m[V_1]+(m-1)[V_2] \ \text{ and } [V]=(m-1)[V_1]+m[V_2]$$
for all $m\ge 1$. 

Indeed, let $V_3[1]$ denote the right mutation of $V_1$ through $V_2$, so that we have an exact triangle
$$V_1\to V_2^{\oplus 2}[1]\to V_3[1]\to\ldots$$
Then $V_3$ is an extension of $V_1$ by $V_3^{\oplus 2}$, so $[V_3]=[V_1]+2[V_2]$. Note that the space $\Ext^*(V_2,V_3)=\Hom(V_2,V_3)$ is $2$-dimensional.
and this property is preserved by the right mutations.
Using this we can check that the part of the helix $(V_2,V_3,V_4,\ldots)$ generated by $(V_2,V_3)$ consists of vector bundles satisfying $[V_{m+1}]=(m-1)[V_1]+m[V_2]$.
Indeed, the equality in $K_0$ follows by induction from the exact triangles
$$V_{m-1}\to V_m^{\oplus 2}\to V_{m+1}\to\ldots$$
Taking into account the fact that $\rk(V_{m-1})<2\rk(V_m)$, we see that $\und{H}^0(V_{m+1})\neq 0$. Since $V_{m+1}$ is an exceptional object, this implies that
it is a sheaf on $X$ (see \cite[Prop.\ 2.10]{KO}). Since it also has positive rank, it has to be a vector bundle (by \cite[Prop.\ 2.9]{KO}).
Similarly, considering left mutations of the pair $(V_1,V_2)$ we find vector bundles $V$ with $[V]=m[V_1]+(m-1)[V_2]$.

It is easy to check that we get the desired $r$ and $d$ this way. Namely, let us write $r=2m-1$ (recall that $r$ is odd).
If $d$ is even we use $n=0$, in which case $\chi(V_1)=2k$ and $\chi(V_2)=2k-2$, so we will get from the above $V$ either
$d=(2k-1)r+1$ or $d=(2k-1)r-1$.
If $d$ is odd we use $n=1$, in which case $\chi(V_1)=2k-1$ and $\chi(V_2)=2k-3$, and so, $d=(2k-2)r\pm 1$.
\end{proof}

\begin{remark}
In the situation of Proposition \ref{biham5-prop}, the dimension of $H^0(X,\om_X^{-1})$ is $5$, so we can expect
that there exists a $5$-dimensional linear space of compatible Poisson brackets on $\PP^{d-1}$ including $q_{d,r}(C)$.
\end{remark} 

\section{Explicit computations}\label{computations-sec}

\subsection{Szeg\"o kernels}

\subsubsection{Case of a bundle with vanishing cohomology}

Let $C$ be an elliptic curve with a fixed nonzero regular differential $\eta$.
Let $V$ be a vector bundle on $C$ such that $H^*(C,V)=0$.
Then there is a unique section called the {\it Szeg\"o kernel} (see e.g., \cite{Fay}),
$$S_V\in H^0(C\times C,V^\vee\boxtimes V(\De))$$
such that $\Res_{\De}(S_V)=\id_V$ (where we use the trivialization of $\om_C$).

\begin{example} Assume that we work over complex numbers, $C=\CC/\Lambda$, and
$V=M$, a nontrivial line bundle of degree zero. We can write $M=\OO_C(a-b)$.
Then one has
$$S_M(x,y)=\zeta(x-y)-\zeta(x-b)+\zeta(y-a)-\zeta(b-a),$$
where $\zeta$ is the Weierstrass zeta function.
We can trivialize the pull-back of $M$ to $\CC$ by the section $\th_{11}(x-b)/\th_{11}(x-a)$,
where $\th_{11}$ is the theta-function with zero at $x=0$. Then with respect to this trivialization,
$$S_M(x,y)=[\zeta(x-y)-\zeta(x-b)+\zeta(y-a)-\zeta(b-a)]\cdot\frac{\th_{11}(x-b)\th_{11}(y-a)}{\th_{11}(x-a)\th_{11}(y-b)}.$$
\end{example}


Note that since $H^*(C,V)=0$, the complex
$$H^0(C-p,V)\rTo{\de_V} H^0(C,V(\infty p)/V)$$
is exact. Here the target can be identified with the quotient $H^0(C,V(\infty p)|_{\infty p})/H^0(C,V|_{\infty p})$, where
$H^0(C,V|_{\infty p})$ is the completion of $V_p$ with respect to the $\fm_p$-adic topology, while 
$$H^0(C,V(\infty p)|_{\infty p})=\hat{V}_p\ot_{\hat{\OO}_{C,p}} K_p,$$
where $K_p$ is the field of fractions of $\hat{\OO}_{C,p}$.

Our goal is to get a formula for $\de_V^{-1}$ in terms of the Szeg\"o kernel $S_V$ (see Lemma \ref{acyclic-Sz-lem} below). 
In fact, for our computations later we will need the case where $V$ is a trivial bundle and the above concept of the Szeg\"o kernel has to
be modified (see Sec.\ \ref{Sz-O-sec}). However, we first consider the case of $V$ with vanishing cohomology since this case is more straightforward. 


We have a natural perfect duality
\begin{equation}\label{V-perfect-duality-eq}
H^0(C,V(\infty p)|_{\infty p})\ot H^0(C,V^\vee(\infty p)|_{\infty p})\to k: B(\phi,f):=\Res_p(\lan\phi, f\ran\cdot\eta).
\end{equation}
Also, we have direct sum decomposition
$$H^0(C,V(\infty p)|_{\infty p})=H^0(C-p,V)\oplus H^0(C,V|_{\infty p}),$$
$$H^0(C,V^\vee(\infty p)|_{\infty p})=H^0(C-p,V^\vee)\oplus H^0(C,V^\vee|_{\infty p}),$$
such that 
$$H^0(C,V|_{\infty p})=H^0(C,V^\vee|_{\infty p})^\perp, \ \ H^0(C,V(\infty p)|_{\infty p})=H^0(C,V^\vee(\infty p)|_{\infty p})^\perp$$
with respect to the above duality.

\begin{lemma}\label{acyclic-Sz-lem}
(i) For any $f\in H^0(C-p,V)$
one has 
$$\Res_{x=p} \lan f(x), S(x,y)\ran=-f(y).$$

\noindent
(ii) One has  
\begin{equation}\label{S-expansion-eq}
S_V|_{\infty p\times C\setminus p}=-\sum_{i\ge 1} \phi_i\otimes f_i,
\end{equation}
where $(\phi_i)$ and $(f_i)$ are dual bases of $H^0(C,V^\vee|_{\infty p})$ and $H^0(C-p,V)$.

\noindent
(iii) There is a well defined linear operator
$$Q_S: H^0(V(\infty p)/V)\to H^0(C-p,V): f\mapsto -\Res_{x=p} \lan f(x), S(x,y)\ran,$$
and we have $Q_S=\de_V^{-1}$.
\end{lemma}

\begin{proof} 
(i) Let us fix a generic $y$ and consider the restriction of $\lan f(x), S(x,y)\ran$ to $C\times y$. It has poles at $x=p$ and $x=y$,
and the residue at $x=y$ is equal to $f(y)$. Thus, the assertion follows from the Residue Theorem.

\noindent
(ii) First, we observe that $S_V|_{\infty p\times C\setminus p}$ lies in 
$$\varprojlim_n H^0(C,V^\vee|_{np})\ot H^0(C-p,V),$$
which can be viewed as a competed tensor product of $H^0(C,V^\vee|_{\infty p})$ and $H^0(C-p,V)$.
The right-hand side of \eqref{S-expansion-eq} also makes sense as an element of this completed tensor product.
Now the assertion follows from (i) and from perfect duality \eqref{S-expansion-eq}.

\noindent
(iii) Note that $Q_S$ is well defined since for regular $f$ the expression $\lan f(x), S(x,y)\ran$ will be regular at $x=p$.
The second assertion follows from (i). 
\end{proof}

\subsubsection{Case of the trivial bundle}\label{Sz-O-sec}

Now let us consider the case $V=\OO_C$. Here Szeg\"o kernel will depend on an extra datum. 
Let $D=p_1+\ldots+p_d$ be a simple divisor on an elliptic curve $C$ (so the points $p_1,\ldots,p_d$ are distinct). 
As before, we fix a trivialization $\eta$ of $\om_C$. We use this trivialization implicitly in formulas with residues.

\begin{definition} We say that $S\in H^0(C\times C,\OO(D)\boxtimes \OO(D)(\De))$ is a {\it left Szeg\"o kernel} for $D$
if we have
\begin{itemize}
\item $\Res_{\De}(S)=1$;
\item $\Res_{D\times C}(S)$ is constant along $D$.
\end{itemize}
If in addition $S(y,x)=-S(x,y)$ then we say that $S$ is a {\it Szeg\"o kernel}.
\end{definition}
 
\begin{example} In the case when $D=p$ has degree $1$, it is easy to check that there is a unique section 
$$S=S_p\in H^0(C\times C,\OO(p)\boxtimes \OO(p)(\De))$$
such that $S(y,x)=-S(x,y)$ and $\Res_{\De}(S)=1$. Hence, it is a Szeg\"o kernel for $p$.
Note that $-S(-x,-y)$ also satisfies these conditions, so we have $S(-x,-y)=-S(x,y)$.
In fact, for an elliptic curve over complex numbers, and $p$ corresponds to the origin, then one has
$$S(x,y)=\zeta(x-y)-\zeta(x)+\zeta(y),$$
where $\zeta$ is the Weierstrass zeta function.
\end{example}


Let $t_i$ be the formal parameter on $C$ at $p_i$ such that $\eta=dt_i$, and let
us consider the vector space
$$V=V_D:=\bigoplus_{i=1}^d k(\!(t_i)\!).$$
We equip $V$ with the nondegenerate pairing 
$$(f,g)=\sum_{i=1}^p \Res_{t_i=0}(fg dt_i).$$
We have the isotropic subspaces 
$$\La:=\bigoplus_{i=1}^d k[\![t_i]\!]\sub V$$
and
$$\OO(C-D)\sub V,$$
where the embedding is given by expanding into Laurent series at $p_1,\ldots,p_d$.
The complex
$$\OO(C-D)\rTo{\de_\OO} V/\La$$
calculates $H^*(C,\OO)$, so $\OO(C-D)\cap \La=\lan 1\ran$ and $\OO(C-D)+\La$ is precisely the codimension $1$ subspace
$$V':=\{f\in V \ |\ \sum_i \Res_{t_i=0}(f dt_i)=0\}.$$ 


We have the following analog of Lemma \ref{acyclic-Sz-lem}.
Let us set 
$$\La'=\{(f_i)_{i=1,\ldots,d}\in \La \ |\ \sum_i f_i(0)=0\}.$$
Note that $\de_\OO$ factors through an embedding
$$\de'_\OO:\OO(C-D)\to V/\La'.$$

\begin{lemma}\label{Sz-O-lem}
Let $f$ be a left Szeg\"o kernel for $D$.

\noindent
(i) For any $f\in \OO(C-D)$
one has 
$$\sum_{i=1}^d\Res_{x=p_i} (f(x)S(x,y))=-f(y).$$


\noindent
(ii) We have a well defined operator 
$$Q'_S: V/\La'\to \OO(C-D): f\mapsto -\sum_{i=1}^d\Res_{x=p_i} (f(x)S(x,y)),$$
such that 
$$Q'_S\de'_\OO(f)=f.$$
Here we view elements of $V$ as functions on a punctured formal neighborhood of $D$.
\end{lemma}
 
\begin{proof}
(i) This immediately follows from the Residue Theorem (for fixed $y$).

\noindent
(ii) Let us first check that $Q'_S$ is well defined. Since $S(x,y)$ has poles of order $1$ at $D$, for $f\in \La'$, one has 
$$\sum_{i=1}^d\Res_{x=p_i} (f(x)S(x,y))=\sum_i f(p_i)\Res_{x=p_i}S(x,y).$$
But by assumption, $\Res_{x=p_i}S(x,y)$ does not depend on $i$, and $\sum_i f(p_i)=0$, so this is zero.

The equality $Q'_S\de'_\OO(f)=f$ follows from (i). 
\end{proof}


\begin{corollary}\label{Cech-Sz-cor} 
For any $f\in \OO(C-D)$ and any lifting $\wt{\de_{\OO}(f)}\in V/\La'$ of $\de_{\OO}(f)\in V/\La$, one has
$$Q'_S\wt{\de_{\OO}(f)}\equiv f \mod \lan 1\ran.$$
For any $g\in V'/\La\sub V/\La$ and any lifting $\wt{g}\in V'/\La'$ of $g$ one has 
$$\de_{\OO}Q'_S(\wt{g})=g.$$
\end{corollary}

\begin{proof}
The first equality follows immediately from Lemma \ref{Sz-O-lem}(ii) since
$$\wt{\de_\OO(f)}=\de'_\OO(f+ c)\in V/\La'$$
for some $c\in k$. 

Given $g\in V'/\La$, we can find $f\in \OO(C-D)$ such that 
$g=\de_\OO(f)$. Now the second equality follows from the first, since $\de_\OO(1)=0$. 
\end{proof}
 
\begin{lemma}\label{alg-for-Sz-lem} 
Let $C$ be an elliptic curve with a divisor $D$ such that either
\begin{enumerate}
\item $D=p$ and $C-p$ is the curve $y^2=P(x)$ in $\AA^2$, where $P$ is a cubic polynomial, or
\item $D=p_1+p_2$ and $C-D$ is the curve $y^2=P(x)$ in $\AA^2$, where $P$ is a quartic polynomial.
\end{enumerate}
As a trivialization of $\om_C$ in both cases we take $\eta=dx/2y$. Then
$$S:=\frac{y_1+y_2}{x_2-x_1}$$
is a Szeg\"o kernel on $C$.
\end{lemma}
 
\begin{proof}
To calculate the residue along the diagonal, we consider the residue of the $2$-form
$$S\cdot\eta_1\wedge\eta_2=\frac{y_1+y_2}{4y_1y_2(x_2-x_1)}\cdot dx_1\wedge dx_2=\frac{y_1+y_2}{4y_1y_2}\cdot dx_1\wedge\frac{d(x_2-x_1)}{x_2-x_1},$$
so the residue is
$$\frac{y_1+y_2}{4y_1y_2}\cdot dx_1|_{\De}=\frac{2y}{4y^2} dx=\eta.$$

Note that $S(y,x)=-S(x,y)$.
Thus, it remains to study the polar part of $S$ near $x_1\in D$. 
In case (1), since $x_1$ has a pole of order $2$ at $D=p$ and $y_1$ has a pole of order $3$, we see that $S$ has a pole of order $1$ at $x_1=p$.
In case (2), let $P(x)=ax^4+\ldots$, where $a\neq 0$. Then we can take $t=1/x$ as a local parameter at both $p_1$ and $p_2$.
In terms of this parameter, $y$ has an expansion  
$$y=\frac{\sqrt{a}}{t^2}+\ldots$$
at $p_1$ (for some choice of $\sqrt{a}$; for $p_2$ it would be a different choice of the square root).
Hence, $\eta$ and $S$ have the expansions (for fixed $x_2,y_2$)
$$\eta(t)=(-\frac{1}{2\sqrt{a}}+\ldots)\cdot dt,$$
$$S(t;x_2,y_2)=\frac{\sqrt{a}/t^2+\ldots}{-1/t+x_2}=-\frac{\sqrt{a}}{t}+\ldots$$
Hence $S$ has a pole of order $1$ and
$$\Res_{(x_1,y_1)=p_1}(S\cdot\eta(x_1,y_1))=\frac{1}{2}.$$
The same calculation works for $p_2$, so we deduce that $S$ is a Szeg\"o kernel for $D$.
\end{proof} 
 
\subsection{Massey product in terms of Szeg\"o kernel}

Now we can present the formula for the Massey product in terms of the Szeg\"o kernel.
Assume $\xi$ is a line bundle of positive degree on $C$, $D\sub C$ a simple divisor.

The multiplication with a Szeg\"o kernel $S=S_D\in H^0(C^2,\OO(D)\boxtimes \OO(D) (\De))$ induces a morphism
\begin{equation}\label{mult-Szego-eq}
{\bigwedge}^2H^0(C,\xi)\rTo{\mu_S}  H^0(C,\xi(D))\ot H^0(C,\xi(D))
\end{equation}
that fits into a commutative diagram
\[\xymatrix{
\bigwedge^2H^0(C,\xi)\ar[r]^{\mu_S}\ar[d]  & H^0(C,\xi(D))\ot H^0(C,\xi(D))\ar[d]\\
H^0(C\times C,\xi\boxtimes \xi)\ar[r]^{S\cdot } & H^0(C\times C,\xi(D)\boxtimes \xi(D)(\De))}
\]
Indeed, this follows easily from the fact that the residue of $S$ along the diagonal is equal to $1$.


\begin{prop}\label{Massey-Sz-prop} 
Let $S$ be a left Szeg\"o kernel for $D$.
Then for $\phi\in H^1(C,\xi^{-1})$ and $s_1,s_2\in \lan\phi\ran^{\perp}\sub H^0(C,\xi)$, one has
$$\lan \phi, MP(s_1,\phi,s_2)\ran=\pm \lan \wt{\phi}\ot\wt{\phi}, \mu_S(s_1\wedge s_2)\ran,$$
where $\wt{\phi}$ is a lifting of $\phi$ to $H^1(C,\xi^{-1}(-D))$;
on the left we use a pairing between $\lan\phi\ran\sub H^1(C,\xi^{-1})$ and $H^0(C,\xi)/\lan s_1,s_2\ran$;
on the right we use the Serre duality between $H^1(C,\xi^{-1}(-D))$ and $H^0(C,\xi(D))$. 
\end{prop}

\begin{proof}
We compute this Massey product using the dg-enhancement given by the Cech resolutions corresponding to the covering by $C-D$ and the formal neighborhood of $D$.
Let us represent $\phi$ by a $1$-cocycle $\phi\in H^0(\xi^{-1}(\infty D)/\xi^{-1})$. Then with the notation of Sec.\ \ref{Sz-O-sec} we have
$$s_1\phi, s_2\phi\in V'/\La\sub V/\La$$
(this follows from the fact that both $s_1\phi$ and $s_2\phi$ have trivial cohomology class in $H^1(C,\OO)$).
Let us choose a lifting of $\phi$ to $\wt{\phi}\in H^0(\xi^{-1}(\infty D)/\xi^{-1}(-D))$.
Then for $i=1,2$, $s_i\wt{\phi}$ is an element of $V'/\La_1$ lifting $s_i\phi$, where
$$\La_1:=\bigoplus_{i=1}^d t_ik[\![t_i]\!]\sub \La.$$
Hence, 
$$f_i:=Q'_\OO(s_i\wt{\phi})$$
is a well defined element of $\OO(C-D)$ satisfying
$$\de_\OO(f_i)=s_i\phi$$
(see Corollary \ref{Cech-Sz-cor}).
Therefore, the dg-recipe for calculating the Massey product gives
$$MP(s_1,\phi,s_2)=f_1s_2-s_1f_2 \mod \lan s_1,s_2\ran.$$
Now we recall the definition of $Q'_\OO$:
$$(f_1s_2)(y)=Q'_\OO(s_1\wt{\phi})s_2=-\tau_x[S(x,y)\wt{\phi}(x)s_1(x)s_2(y)],$$
where
$$\tau_x=\sum_{i=1}^d \Res_{x=p_i}.$$
Similarly,
$$(s_1f_2)(y)=-\tau_x[S(x,y)\wt{\phi}(x)s_2(x)s_1(y)].$$
Hence,
\begin{align*}
&\lan \phi,MP(s_1,\phi,s_2)\ran=\tau_y[\wt{\phi}(y)\cdot MP(s_1,\phi,s_2)(y)]=-\tau_x\tau_y[\wt{\phi}(x)\wt{\phi}(y)\cdot S(x,y)(s_1(x)s_2(y)-s_2(x)s_1(y))]\\
&=-\lan \wt{\phi}\ot \wt{\phi}, S\cdot (s_1\wedge s_2)\ran.
\end{align*}
\end{proof}

Now we can give a formula for the Poisson bracket on $\PP H^1(C,\xi^{-1})$ in terms of the Szeg\"o kernel and certain auxiliary data which exists
in some examples.

\begin{theorem}\label{Szego-Poisson-thm}
Let $\xi$ be a line bundle of positive degree on an elliptic curve $C$, $D$ a simple effective divisor on $C$, $S$ a left Szeg\"o kernel for $D$.
Suppose there exist another effective
divisor $E$ on $C$ and linear operators 
$$A,B:H^0(C,\xi)\to H^0(C,\xi(D+E))$$ 
such that for any $s_1,s_2\in H^0(C,\xi)$ one has
$$S\cdot (s_1\we s_2)+s_1\ot A(s_2)-s_2\ot A(s_1)+B(s_2)\ot s_1-B(s_1)\ot s_2\in H^0(C,\xi)\ot H^0(C,\xi).$$
Then for nonzero $\phi\in H^1(C,\xi^{-1})$ and $s_1,s_2\in \lan\phi\ran^\perp$, one has
$$\Pi_{\lan\phi\ran}(s_1\we s_2\ran)=\pm\lan \phi\ot \phi, S\cdot (s_1\we s_2)+s_1\ot A(s_2)-s_2\ot A(s_1)+B(s_2)\ot s_1-B(s_1)\ot s_2\ran.$$
\end{theorem}

\begin{proof} By Lemma \ref{Poisson-Massey-lem} and Proposition \ref{Massey-Sz-prop}, we have
$$\Pi_{\lan\phi\ran}(s_1\we s_2\ran)=\pm\lan \phi, MP(s_1,\phi,s_2)\ran=\pm\lan \wt{\phi}\ot\wt{\phi}, S\cdot (s_1\wedge s_2)\ran,$$
where $\wt{\phi}$ is a lifting of $\phi$ to $H^1(C,\xi^{-1}(-D-E))$ (note that in the right-hand side we can replace $\wt{\phi}$ by the induced lifting of $\phi$ to $H^1(C,\xi^{-1}(-D))$).
Now we observe that
$$\lan \wt{\phi}, s_i\ran=\lan \phi, s_i\ran=0$$
for $i=1,2$, since $s_i\in \lan\phi\ran$.
Hence,
$$\lan \wt{\phi}\ot\wt{\phi}, S\cdot (s_1\wedge s_2)\ran=\lan \wt{\phi}\ot\wt{\phi}, S\cdot (s_1\we s_2)+s_1\ot A(s_2)-s_2\ot A(s_1)+B(s_2)\ot s_1-B(s_1)\ot s_2\ran.$$
Finally, we can replace $\wt{\phi}\ot\wt{\phi}$ with $\phi\ot \phi$ since the second argument of the pairing lies in $H^0(C,\xi)\ot H^0(C,\xi)$.
\end{proof}

\subsection{Odesskii-Wolf compatible brackets}\label{OW-sec}

Here we are going to prove that $9$ compatible Poisson brackets on projective spaces constructed in Example \ref{OW-ex} coincide with those constructed
by Odesskii-Wolf in \cite{OW}. Note that for this it is enough to check the equality between two brackets for a generic value of parameters in the linear family (resp., a generic
anticanonical divisor in the Hirzebruch surface).

\subsubsection{Even case}

Let us first consider the case of brackets containing $q_{2k,1}$.
This corresponds to considering anticanonical divisors in $X=\PP(\OO\oplus \OO(2))$. 
Let $p:X\to \PP^1$ be the natural projection
We denote by $(t_0:t_1)$ the homogeneous coordinates on $\PP^1$ and by $(x_0:x_1)$ the fiberwise homogeneous coordinates on $X$,
where $x_0$ is a section of $\OO_X(1)$ and $x_1$ is a section of $p^*\OO(2)(1)$.
Since $\om_X^{-1}=p^*\OO(4)(2)$, we have
$$H^0(X,\om_X^{-1})=\k\cdot x_1^2\oplus p^*H^0(\PP^1,\OO(2))\cdot x_1x_0\oplus p^*H^0(\PP^1,\OO(4))x_0^2.$$
Thus, a generic anticanonical divisor $C\sub X$ is given by the equation
$$x_1^2=f_2(t_0,t_1)x_1x_0+f_4(t_0,t_1)x_0^2,$$
where $f_2$ is homogeneous of degree $2$ and $f_4$ is homogeneous of degree $4$.
Note that $x_0\neq 0$ on $C$, so it gives a trivialization of $\OO_X(1)|_C$.

Let us denote by $D\sub C$ the divisor $t_0=0$. Then we can use 
$$t:=\frac{t_1}{t_0},  \ \ x:=\frac{x_1}{x_0t_0^2}$$
as affine coordinates on $C-D$ satisfying the equation
$$x^2=Q(t)x+P(t),$$
where $Q(t)$ has degree $\le 2$ and $P$ has degree $\le 4$.
Note that the space $\FF_{ev}$ in \cite[Sec.\ 2.1]{OW} is precisely the space of functions on $C-D$.

We can rewrite the equation of $C-D$ in the form
$$(x-Q(t)/2)^2=P(t)+Q(t)^2/4.$$
Hence, by Lemma \ref{alg-for-Sz-lem}, 
\begin{equation}\label{Sz-Q-formula}
S=\frac{x_1-Q(t_1)/2+x_2-Q(t_2)/2}{t_1-t_2}
\end{equation}
is a Szeg\"o kernel for $D$.

For $k\ge 1$, we consider the line bundle
$$\xi_{2k}:=p^*\OO(k)(1)|_C\simeq p^*\OO(k)\simeq \OO_C(kD),$$
where we use the trivialization of $\OO_X(1)|_C$ given by $x_0$.
The restriction map on spaces of global sections
$$H^0(X,p^*\OO(k)(1))\to H^0(C,\xi_{2k})$$
is an isomorphism, and sends the basis 
$$(t_1^it_0^{k-i}x_0)_{i\le k}, (t_1^jt_0^{k-2-j}x_1)_{j\le k-2}$$
to the basis $(t^i)_{i\le k}, (t^jx)_{j\le k-2}$ of $H^0(C,\xi_{2k})$.
Thus, we can identify this space with the subspace $\FF_{2k}\sub \FF_{ev}$ defined in \cite[Sec.\ 2.1]{OW}.

Recall that Odesskii-Wolf \cite{OW} define a derivation $\cD$ on $\FF_{ev}=\OO(C-D)$ by
$$\cD(t)=2x-Q(t), \ \ \cD(x)=P'(t)+Q'(t)x.$$
Note that the fact that $\cD$ descends to a well defined derivation of $\OO(C-D)$ becomes clear if we rewrite it as
$$\cD=\frac{\pa F}{\pa x}\pa_t - \frac{\pa F}{\pa t}\pa_x,$$
where $F=x^2-Q(t)x-P(t)$ is the defining equation of $C-D$.
Also, it is easy to check that 
$$\cD(H^0(C,\OO(kD)))\sub H^0(C,\OO((k+1)D)).$$

Now the Poisson bracket from \cite{OW} on $\PP H^0(C,\xi_{2k})^*\sim \PP \FF_{2k}^*$ (depending linearly on the coefficients of $Q$ and $P$) can be rewritten as
\begin{equation}\label{OW-even-bracket-formula}
\lan \Pi_{OW,\phi}, s_1\we s_2\ran=\lan \phi\ot \phi, 2k\cdot S\cdot (s_1\we s_2)+s_1\ot \cD(s_2)+\cD(s_2)\ot s_1-s_2\ot \cD(s_1)-\cD(s_1)\ot s_2\ran,
\end{equation}
where $\phi\in H^0(C,\xi_{2k})^*$, $s_1,s_2\in\lan\phi\ran^\perp$, and $S$ is given by \eqref{Sz-Q-formula}. Note that a part of the statement (that is proved in \cite{OW} by
a direct computaton) is that the second argument in the pairing in the right-hand side of \eqref{OW-even-bracket-formula} lies in $H^0(C,\xi_{2k})\ot H^0(C,\xi_{2k})$.
Therefore, using Theorem \ref{Szego-Poisson-thm} (with $E=0$) we see that our construction of compatible brackets agrees with that of \cite{OW} in this case.

\begin{prop}\label{OW-even-prop}
The $9$ compatible Poisson brackets on $\PP \FF_{2k}^*$ given in \cite{OW} are linearly independent and the corresponding
$9$-dimensional subspace of compatible brackets coincides with the one coming from Example \ref{OW-ex} for $n=2$.
\end{prop}

\begin{proof} We checked the compatibility between two constructions. It remain to prove linear independence.
Let us consider the group 
$$G=\GL_2\rtimes \Aut(\OO_{\PP^1}\oplus\OO_{\PP^1}(2)).$$
It acts on the Hirzebruch surface $X$ and the relevant line
bundles are $G$-equivariant, so the kernel of the linear map
$$H^0(X,\om_X^{-1})\to H^0(\PP \FF_{2k}^*, {\bigwedge}^2T)$$
is $G$-invariant. But it is easy to see that the only nonzero proper $G$-subrepresentations of 
$H^0(X,\om_X^{-1})$ are
$$p^*H^0(\PP^1,\OO(2))\cdot x_1x_0\oplus p^*H^0(\PP^1,\OO(4))x_0^2 \ \text{ and } p^*H^0(\PP^1,\OO(4))x_0^2$$
(for this in addition to $\GL_2$ we use automorphisms $x_0\mapsto x_0$, $x_1\mapsto Q(t)x_0$).
Thus, it is enough to check that our map is nonzero on $p^*H^0(\PP^1,\OO(4))x_0^2$.
Therefore, it suffices to check that the image of $\lan x_1^2, t_0^4x_0^2\ran$ is $2$-dimensional.

For this we apply formulas from \cite[Sec.\ 2.2]{OW} to compute the bracket $\{\cdot,\cdot\}_{a_0}$ associated with the anticanonical divisor $C_{a_0}$ given by
$$x_1^2=a_0t_0^4x_0^2$$
(which corresponds in the notation of \cite{OW} to $g^2=a_0$)
and to check that the constant and linear terms in $a_0$ are linearly independent.

Let us consider the linear forms on $\PP\FF_{2k}^*$ (which we view as elements of $H^0(C,\xi_{2k})$, 
$$\ell_1=1, \ \ \ell_2=t, \ \ \ell_3=x.$$
Then using  formulas from \cite[Sec.\ 2.2]{OW} we get
$$\{\frac{\ell_1}{\ell_3},\frac{\ell_2}{\ell_3}\}_{a_0}=-2k\frac{\ell_1}{\ell}_3+a_0\cdot 2k\frac{\ell_1^3}{\ell_3^3}.$$
Hence, we get the required independence.
\end{proof}

\subsubsection{Odd case}

Now we consider the situation of Example \ref{OW-ex} for anticanonical divisors in $X=\PP(\OO\oplus \OO(1))$. 
This time we have fiberwise homogeneous coordinates $x_0\in \OO_X(1)$ and $x_1\in p^*\OO(1)(1)$.
We have $\om_X^{-1}=p^*\OO(3)(2)$, so
$$H^0(X,\om_X^{-1})=p^*H^0(\PP^1,\OO(1))\cdot x_1^2\oplus p^*H^0(\PP^1,\OO(2))\cdot x_1x_0\oplus p^*H^0(\PP^1,\OO(3))x_0^2.$$
Thus, a generic anticanonical divisor $C$ is given by the equation
$$(t_1+ct_0)x_1^2=f_2(t_0,t_1)x_1x_0+f_3(t_0,t_1)x_0^2,$$
where $\deg(f_2)=2$, $\deg(f_3)=3$.
The open affine subset $U\sub C$ given by $t_0x_0\neq 0$ has the algebra of functions generated by $t=t_1/t_0$ and $x=x_1/(x_0t_0)$ subject to the relation
$$(t+c)x^2=Q(t)x+P(t),$$
where $\deg Q\leq 2$ and $\deg P\le 3$.
This algebra is precisely $\FF_{od}$ from \cite[Sec.\ 2.1]{OW}.
 
As before we consider the line bundle $p^*\OO(k)(1)$ on $X$ and its restriction to $C$,
$$\xi_{2k+1}:=p^*\OO(k)(1)|_C.$$
The section $t_0^kx_0$ trivializes this line bundle over $U$, so that the basis of global sections of $p^*\OO(k)(1)$
restricts to the functions 
\begin{equation}\label{odd-case-sections-line-bundle-eq}
(t^i)_{i\le k}, (t^jx)_{j\le k-1}.
\end{equation}
Thus, we have an identification of $H^0(C,\xi_{2k+1})$ with the space $\FF_{2k+1}\sub \FF_{od}$ from \cite{OW}.

As in the even case, Odesskii-Wolf define a derivation $\cD$ on $\FF_{od}=\OO(U)$ by
$$\cD(t)=2(t+c)x-Q(t), \ \ \cD(x)=P'(t)+Q'(t)x-x^2.$$
Further, they define the quadratic Poisson bracket on $\FF_{2k+1}$ (depending linearly on the coefficients of $Q$ and $P$) which induces a Poisson bracket on 
$\PP\FF_{2k+1}^*=\PP H^0(C,\xi_{2k+1})^*$ given by
$$\lan \Pi_{OW,\phi}, s_1\we s_2\ran=\lan \phi\ot \phi, (2k+1)\cdot S\cdot (s_1\we s_2)+s_1\ot \cD(s_2)+\cD(s_2)\ot s_1-s_2\ot \cD(s_1)-\cD(s_1)\ot s_2\ran,$$
where $S$ is given by
$$S=\frac{(t_1+c)x_1-Q(t_1)/2+(t_1+c)x_2-Q(t_2)/2}{t_1-t_2}.$$

To understand this formula let us consider the divisor $D\sub C$ given by $t_0=0$. Then $U\sub C-D$ and the complement consists of one point $q$ where $t_1+ct_0=0$
and $x_0=0$. It is easy to see that $C-D$ is affine and the algebra of functions $\OO(C-D)$ is the subring of $\OO(U)$ generated by $t$ and $z:=(t+c)x$.
Thus, $C-D$ is the plane curve given by the equation
$$z^2=Q(t)z+(t+c)P(t).$$
Now Lemma \ref{alg-for-Sz-lem} shows that $S$ is a Szeg\"o kernel for the divisor $D$ on $C$.
 
On the other hand, 
since $x\in \OO(D+q)$ and has a pole of order $1$ at $q$, looking at the basis \eqref{odd-case-sections-line-bundle-eq} we see that
$H^0(C,\xi_{2k+1})=\FF_{2k+1}$ gets identified with the subspace $H^0(C,\OO(kD+q))\sub \OO(U)$.
It is easy to check that 
$$\cD(H^0(C,\OO(kD+q)))\sub H^0(C,\OO((k+1)D+2q)).$$ 
Thus, applying Theorem \ref{Szego-Poisson-thm} (with $E=q$) we again deduce the agreement of our construction of compatible Poisson brackets with that
of \cite{OW}.

\begin{prop}\label{OW-odd-prop}
The $9$ compatible Poisson brackets on $\PP \FF_{2k+1}^*$ given in \cite{OW} are linearly independent and the corresponding
$9$-dimensional subspace of compatible brackets coincides with the one coming from Example \ref{OW-ex} for $n=1$.
\end{prop}

\begin{proof} It remains to check that the map
$$H^0(X,\om_X^{-1})\to H^0(\PP \FF_{2k+1}^*, {\bigwedge}^2T)$$
is injective.
As before, we use the fact that the kernel is invariant under $\GL_2\rtimes \Aut(\OO_{\PP^1}\oplus\OO_{\PP^1}(1))$,
so it is enough to check that the image of $p^*H^0(\PP^1,\OO(3))x_0^2$ is nonzero.
Hence, it suffices to consider the bracket $\{\cdot,\cdot\}_{a_0}$ corresponding to the anticanonical divisor
$$tx^2=a_0$$
and check that the constant and linear terms in $a_0$ are linearly independent.
 
Let us consider the linear forms on $\PP\FF_{2k+1}^*$ (which we view as elements of $H^0(C,\xi_{2k+1})$, 
$$\ell_1=1, \ \ \ell_2=t, \ \ \ell_3=x.$$
Then using  formulas from \cite[Sec.\ 2.3]{OW} we get
$$\{\frac{\ell_1}{\ell_3},\frac{\ell_2}{\ell_3}\}_{a_0}=
-2\frac{\ell_1}{\ell_3}-(2k-1)\frac{\ell_2}{\ell_3}+a_0(2k+1)\frac{\ell_1^3}{\ell_3^3},$$
so we get the required linear independence.
\end{proof}

\end{document}